\documentclass[12pt]{article}

\usepackage[english]{babel}
\usepackage[cp1251]{inputenc}
\usepackage[T2A]{fontenc}

\usepackage{amssymb}
\usepackage{amsmath}
\usepackage{amsthm} 
\usepackage{amsfonts}
\usepackage{graphicx,geometry}
\usepackage[title]{appendix}
\usepackage{cite}
\usepackage{float}
\usepackage{xcolor}
\usepackage[font=small,labelfont=bf]{caption}

\usepackage{tikz}
\usetikzlibrary{arrows,decorations.pathmorphing,backgrounds,positioning,fit,petri,patterns}

\renewcommand{\leq}{\leqslant}

\DeclareFontEncoding{LS1}{}{}
\DeclareFontSubstitution{LS1}{stix2}{m}{n}
\newcommand{\wideangleup}{\DOTSB\wideangleupbin}
\DeclareRobustCommand{\wideangleupbin}{
  \mathbin{\text{\usefont{LS1}{stix2frak}{m}{n}\symbol{"60}}}
}

\DeclareMathOperator{\sign}{sgn}
\DeclareMathOperator{\supp}{supp}
\DeclareMathOperator{\dist}{dist}
\DeclareMathOperator{\inter}{int}

\newtheorem{theorem}{Theorem}
\newtheorem{lemma}{Lemma}
\newtheorem{proposition}{Proposition}
\newtheorem{corollary}{Corollary}

\theoremstyle{definition}
\newtheorem{definition}{Definition}
\newtheorem{remark}{Remark}

\newenvironment{enbibliography}{\vspace{-0.5cm}\begin{thebibliography}}{\end{thebibliography}}

\usepackage[title]{appendix}

\begin{document} 
\title{Kru\v{z}kov-type uniqueness theorem
for the chemical flood conservation law system with local vanishing viscosity admissibility}
\author{Sergey Matveenko\footnote{ Chebyshev Laboratory, St. Petersburg State University, 14th Line V.O., 29, Saint Petersburg 199178 Russia. E-mail: matveis239@gmail.com.}, Nikita Rastegaev\footnote{St. Petersburg Department of Steklov Mathematical Institute
of Russian Academy of Sciences, 27 Fontanka, 191023, St. Petersburg, Russia. E-mail: rastmusician@gmail.com.}}
\renewcommand{\today}{}
\maketitle
\abstract{
We study the uniqueness of solutions of the initial-boundary value problem in the quarter-plane for the chemical flood conservation law system in the class of piece-wise $\mathcal C^1$-smooth functions under certain restrictions. The vanishing viscosity method is used locally on the discontinuities of the solution to determine admissible and inadmissible shocks. The Lagrange coordinate transformation is utilized in order to split the equations. The proof of uniqueness is based on an entropy inequality similar to the one used in the classical Kru\v{z}kov's theorem.
}

\section{Introduction}
We study the uniqueness of solutions of the conservation law system
\begin{equation}\label{eq:main_system_chem_flood}
\begin{cases} 
s_t + f(s, c)_x = 0, \\
(cs + a(c))_t + (cf(s,c))_x  = 0.
\end{cases}
\end{equation}
This system is often used to describe the chemical flood of oil reservoir in enhanced oil recovery methods. Here $(x,t)\in\mathbb{R}_+^2$, $s$ is the saturation of the water phase, $c$ is the concentration of the chemical agent dissolved in water, $f$ denotes the fractional flow function, usually S-shaped after Buckley--Leverett \cite{BL}, and $a$ describes the adsorption of the chemical agent on the rock, usually concave like the classical Langmuir curve (see Fig.~\ref{fig:BL_ads}).

We study the solutions of the initial-boundary value problem
\begin{align}
\label{eq:Initial_boundary_problem}
\begin{split}
&s(x,0) = s^x_0(x), \quad c(x,0) = c^x_0(x), \quad x\geqslant 0,\\
&s(0,t) = s^t_0(t), \quad c(0,t) = c^t_0(t),  \quad t\geqslant 0,
\end{split}
\end{align}
and under certain restrictions on the parameters of the problem and the class of solutions we prove the uniqueness theorem, that is we prove that two different solutions from the described class with the same initial-boundary data could not exist. 

Note that the problem describing constant injection into a homogeneously filled reservoir, for example
\begin{align*}
&s_0^x(x) = c_0^x(x) = 0, \quad x\geqslant 0, \\
&s_0^t(t) = c_0^t(t) = 1, \quad t\geqslant 0,
\end{align*}
is equivalent to the Riemann problem
\begin{equation*}
    (s,c)(x,0)=
    \begin{cases}
        (1,1),& \text{if } x\leq 0,\\
        (0,0),& \text{if } x>0.
    \end{cases}
\end{equation*}
The Riemann problem for the system \eqref{eq:main_system_chem_flood} was studied in \cite{JnW} and solutions for it are known. The uniqueness of vanishing viscosity solutions for it was also considered in \cite{Shen}. But the conditions \eqref{eq:Initial_boundary_problem} also cover a lot of more complicated problems, including the slug injection problem considered in \cite{PiBeSh06}. We utilize the Lagrange coordinate transformation described in that paper to split the equations and prove a uniqueness theorem similar in proof to the well-known Kru\v{z}kov's theorem \cite{Kruzhkov} for the general case including more complex problems involving, for example, more than one slug of chemical agent or tapering \cite{Tapering}.
When considering a variable injection such as slug injection, it makes more sense for practical purposes to fix the injection profile at $x=0$ instead of just fixing the initial data on the whole $x$ axis. That's why we work with the initial-boundary value problem in the form \eqref{eq:Initial_boundary_problem} instead of solving an arbitrary Cauchy problem.

As noted by Wagner \cite{Wa87} in the context of gas dynamics, the Lagrange coordinate transformation ceases to be a one-to-one mapping in the presence of vacuum states (in our context they correspond to zero flow states $s = f(s, c) = 0$). Therefore, a significant portion of this paper is dedicated to the study of the zero flow area of our solution in order to exclude it when changing coordinates. An additional assumption on the behavior of $|c_x|$ near the shocks is required for the proposed scheme of proof in that portion (see (W2) in Definition \ref{def:solution}).

In order to determine the admissible solutions, we follow the paradigm of the classical work by Oleinik \cite{Oleinik} and require our weak solutions to have a locally finite number of shocks with a certain admissibility condition on each of them. We apply the local variant of the vanishing viscosity condition introduced in \cite{Bahetal} (see (W4) in Definition \ref{def:solution}). On certain types of shocks this condition is equivalent to the classical Oleinik E-condition, and we use this equivalence to simplify the transfer of our admissibility condition into the Lagrange coordinates. This form of admissibility is convenient for practical applications, since for any analytically constructed solution it is easier to verify the set of conditions given in Definition \ref{def:solution}, than it is to verify the classical vanishing viscosity condition. In the future we aim to derive some or most of the conditions (W1)--(W4) in Definition \ref{def:solution} from the classical vanishing viscosity condition.

The paper has the following structure. Sect.~\ref{sec:restrictions} lists all restrictions we place on the parameters of the problem, i.e. on the initial-boundary conditions, on the flow function $f$ and on the adsorption function $a$. Sect.~\ref{sec:admissibility} defines the class of admissible solutions and derives the travelling wave dynamic system for the dissipative system, which helps analyze the set of admissible shocks. Sect.~\ref{sec:Lagrange} describes the Lagrange coordinate transformation. First, the zero flow area is analyzed. Then, the weak formulation of the equations is used to derive the transformed equations in new coordinates. The qualities of the new flow function are also derived in this section. Sect.~\ref{sec:entropy} describes how the shocks in original coordinates map into shocks in the Lagrange coordinates. The admissibility conditions are also transfered to the new coordinates and the Kru\v{z}kov-type integral entropy inequalities are derived for the transformed system. Sect.~\ref{sec:uniqueness} contains the proof of the uniqueness theorem based on the entropy inequalities obtained in the previous section. Finally, Sect.~\ref{sec:discussions} contains discussions on the result as well as the list of possible future generalizations and developments of the proposed techniques. 

\section{Restrictions on problem parameters}
\label{sec:restrictions}

\subsection{Restrictions on the initial-boundary data}
The ultimate goal is to prove the uniqueness theorem for arbitrary initial and boundary data, but in this paper we will limit ourselves with the following restrictions on the functions from \eqref{eq:Initial_boundary_problem}:
\begin{itemize}
    \item[(S1)] $s^x_0(x) = 0$ for all $x\geqslant x^0$ for a fixed $x^0\in[0,+\infty]$;
    \item[(S2)] $s^x_0(x) \geqslant \delta^0 > 0$ for all $0 \leqslant x < x^0$;
    \item[(S3)] $s^t_0(t) \geqslant \delta^0 > 0$ for all $t\geqslant 0$.
\end{itemize}
Fully lifting these assumptions would create complications in the studying of the zero flow area in Sect.~\ref{sec:zero-flow} and require a lot of technical considerations that we leave for future works on the subject.

\subsection{Restrictions on the flow function}
The following assumptions (F1)--(F4) for the fractional flow function $f$ are considered (see Fig.~\ref{fig:BL_ads}a for an example of function~$f$). 
\begin{enumerate}
    \item[(F1)] $f\in \mathcal C^2([0,1]^2)$; $f(0, c)=0$, $f(1, c)= 1$ for all $c\in[0,1]$;
    \item[(F2)] $f_s(s, c)>0$ for $0<s<1$, $0 \leq c \leq 1$;  $f_s(0,c)=f_s(1,c)=0$  for all $c\in[0,1]$;
    \item[(F3)] $f$ is $S$-shaped in $s$: for each $c \in [0,1]$ function $f(\cdot,c)$ has a unique point of inflection $s^I =  s^I(c) \in (0, 1)$, such that 
    \begin{itemize}
        \item[-] $f_{ss}(s, c)>0$ for $0<s<s^I(c)$,
        \item[-] $f_{ss}(s, c)<0$ for $s^I(c)<s<1$; 
    \end{itemize}
    \item[(F4)] $f$ is monotone in $c$, i.e. one of the variants holds:
    \begin{itemize}
        \item[(F4.1)] 
        $f_c(s, c)<0$ for all $s, c \in (0,1)$;
        \item[(F4.2)] $f_c(s, c)>0$ for all $s, c \in (0,1)$.
    \end{itemize}
\end{enumerate}

\begin{remark}
We will assume (F4.1) for all proofs and figures. The proof of the uniqueness theorem does not change if the monotonicity of $f$ with respect to $c$ is reversed.
\end{remark}

Sometimes in applications function $f$ is supported on $(s_{\min}, s_{\max})$ where $0<s_{\min}<s_{\max}<1$, i.e. in (F1) we have $f(s_{\min}, c) = 0$, $f(s_{\max}, c) = 1$. If $s_{\min}$ and $s_{\max}$ do not change with $c$, then we propose a change of variable
\[
\widetilde{s} = s - s_{\min}, \quad \widetilde{a}(c) = a(c) + s_{\min} c,
\]
which shifts $s_{\min}$ to zero. It is easy to see that the value of $s_{\max}$ has no effect on the viability of this paper's proofs, therefore this case is covered.

The S-shaped assumption (F3) comes naturally from oil recovery applications where the flow function is obtained as the ratio of the water mobility and the total mobility of the phases. When the phase mobilities are given by power functions (the so-called Corey-type model) it was proven in \cite{Castaneda} that the resulting flow function is S-shaped. For a more general sufficient condition for an S-shaped flow function applicable to other mobility models see \cite{RastS-Shaped}.

\subsection{Restrictions on the adsorption function}
The adsorption function $a=a(c)$ satisfies the following assumptions (see Fig.~\ref{fig:BL_ads}b for an example of function~$a$):
\begin{itemize}
    \item[(A1)] $a \in \mathcal C^2([0,1])$, $a(0) = 0$;
    \item[(A2)] $a_c(c)>0$ for $0<c<1$;
    \item[(A3)] $a_{cc}(c)<0$ for $0<c<1$.
\end{itemize}

\begin{figure}[htbp]
    \centering
    \includegraphics[width=0.4\textwidth]{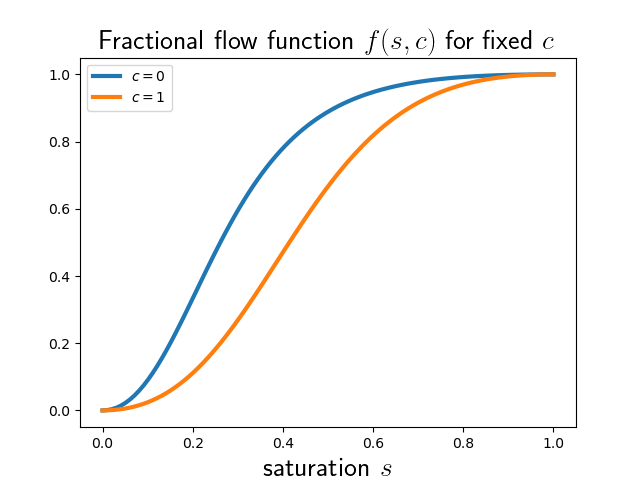}
    \includegraphics[width=0.4\textwidth]{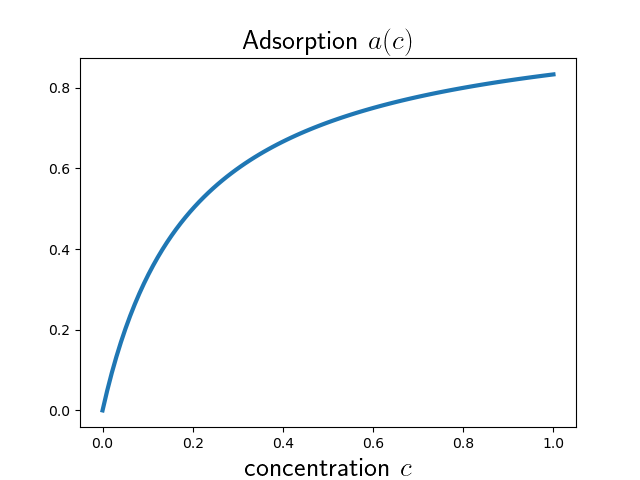}\\
    (a)\qquad\hfil\qquad  
    (b)\hfil
    \caption{Examples of (a) flow function $f(s,c)$;
    (b) adsorption function $a$.}
    \label{fig:BL_ads}
\end{figure}

\section{Admissible solutions of chemical flood system}
\label{sec:admissibility}

\subsection{Admissible weak solutions}
It is well known that without sufficient restrictions the hyperbolic conservation laws admit multiple weak solutions. Some conditions must be imposed to obtain a uniqueness theorem. There are several approaches to formulating such restrictions. The most well-known methods for a single scalar conservation law are as follows:
\begin{itemize}
    \item \emph{Vanishing viscosity method.} The admissible weak solutions are the limits of smooth solutions of a dissipative system. This approach was used in \cite{Kruzhkov} to obtain global integral entropy inequalities.
    \item \emph{Shock conditions.} The admissible weak solutions are required to be a classical solution almost everywhere, except for a finite number of shocks (jump discontinuities) with a certain entropy condition held locally on the shocks. The famous Oleinik E-condition \cite{Oleinik} is one of the most well-known shock conditions. 
\end{itemize}
There are strong connections between these different methods, especially in the case of self-similar solutions of the Riemann problem. In this work we propose a mix of these two approaches:
\begin{itemize}
    \item \emph{Local vanishing viscosity method.} The admissible weak solutions are required to be a classical solution almost everywhere except for a locally finite number of shocks (jump discontinuities) with a local version of the vanishing viscosity condition on the shocks. 
\end{itemize}

\begin{definition}\label{def:solution}
We call $(s, c)$ a piece-wise $\mathcal C^1$-smooth weak solution of \eqref{eq:main_system_chem_flood} with vanishing viscosity admissible shocks and locally bounded ``variation'' of $c$, if:
\begin{itemize}
    \item[(W1)] Functions $s$ and $c$ are continuously differentiable everywhere, except for a locally finite number of $\mathcal C^1$-smooth curves, where one or both of them have a jump discontinuity.
    \item[(W2)] For any compact $K$ away from the axes, the derivative $|c_x(x,t)| < C_K$ is uniformly bounded for all $(x,t) \in K$ not on jump discontinuities.
    \item[(W3)] Functions $s$ and $c$ satisfy \eqref{eq:main_system_chem_flood} in a classical sense inside the areas, where they are continuously differentiable.
    \item[(W4)] On every discontinuity curve $\Gamma$ given by $\gamma(t)$ at any point $(\gamma(t_0), t_0)$ the jump of $s$ and $c$ 
    \[
    s^\pm = s(\gamma(t_0)\pm 0, t_0), \quad c^\pm = c(\gamma(t_0)\pm 0, t_0)
    \]
    with velocity $v = \gamma_t(t_0)$ could be obtained as a limit as $\varepsilon\to 0$ of travelling wave solutions
    \[
    s(x,t) = \mathbf{s}\Big(\frac{x - vt}{\varepsilon}\Big), \quad c(x,t) = \mathbf{c}\Big(\frac{x - vt}{\varepsilon}\Big)  
    \]
    of the dissipative system
\begin{equation}\label{eq:main_system_dissipative}
\begin{cases} 
s_t + f(s, c)_x = \varepsilon s_{xx}, \\
(cs + a(c))_t + (cf(s,c))_x  = \varepsilon (c s_x)_x + \varepsilon c_{xx},
\end{cases}
\end{equation}
with boundary conditions
\[
\mathbf{s}(\pm\infty) = s^\pm, \quad \mathbf{c}(\pm\infty)  = c^\pm.
\]
\end{itemize}
\end{definition}
In this paper we only consider piece-wise $\mathcal C^1$-smooth weak solutions with vanishing viscosity admissible shocks and locally bounded ``variation'' of $c$, therefore, from now on we will just call them solutions for brevity.

Note, that \eqref{eq:main_system_dissipative} differs slightly from \cite[(4.8)]{JnW}, but if we compare the admissible shocks for these systems following the transformations and analysis of sections \ref{sec:sec2-Hopf} and \ref{sec:sec2-nullclines} below step-by-step for both systems, it is clear that they yield exactly the same set of admissible shocks. Moreover, if we consider the system \cite[(3)]{Bahetal}, which accounts for capillary pressure, polymer diffusion and dynamic adsorption, we see that it describes the same set of admissible shocks regardless of the values of small coefficients or capillary pressure function. Therefore, among all of these equivalent variants we settled on the system \eqref{eq:main_system_dissipative}, since it results in a simpler looking dynamic system \eqref{eq:dyn_sys_cap_diff} for travelling wave solutions.

Note also, that the condition (W4) is not the classical form of vanishing viscosity condition. Usually (e.g. in the original Kru\v{z}kov's theorem) the weak solution is assumed to be a limit of smooth solutions of \eqref{eq:main_system_dissipative} as a whole. Instead, we only require that shocks locally behave similar to primitive travelling wave solutions of \eqref{eq:main_system_dissipative}. This form of admissibility is easier to verify when constructing weak solutions via the characteristics method. We believe it is possible to derive the condition (W4) from the classical form of vanishing viscosity with (W1), and we intend to formalize this derivation in the future generalizations of the current result.

Moreover, it is clear that the requirement (W2) and the limitation on the number of shocks in (W1) further restrict the class of problems we are able to consider. 
The assumption (W2) is especially strict. We've chosen it to be more clear, though in practice we only use it in subsection \ref{sec:zero-flow}, where it is sufficient for some integrals of $|c_x|$ along characteristic curves to be bounded on compacts. Ultimately, with some complicated modifications of the proofs, it should be possible to weaken (W2) as follows:
\begin{itemize}
    \item[(W2*)] For any compact $K$ away from the axes, the integral $\int_{t_1}^{t_2}|c_x(x,t)|\, dt < C_K$ is uniformly bounded for all $(x,t_1), (x,t_2)\in K$.
\end{itemize}
We have hope that these properties could be either weakened further or proved a priori for a wide class of initial-boundary conditions. However, we leave such considerations for future works and impose them as assumptions for now. 

\subsection{Travelling wave dynamic system}
\label{sec:sec2-Hopf}
The assumption (W4) for the solution is that shocks are admissible if and only if they could be obtained as a limit of travelling wave solutions for a system with additional dissipative terms as these terms tend to zero. In this section we analyze such travelling wave solutions and derive a dynamic system that describes them.

Consider a shock between states $(s^-, c^-)$ and $(s^+, c^+)$ moving with velocity $v$. In order to check if it is admissible, we are looking for a travelling wave solution 
\[
s(x,t) = s\Big(\frac{x - vt}{\varepsilon}\Big), \quad c(x,t) = c\Big(\frac{x - vt}{\varepsilon}\Big)
\]
for the dissipative system \eqref{eq:main_system_dissipative}
satisfying the boundary conditions 
\[
s(\pm\infty) = s^\pm, \quad c(\pm\infty)  = c^\pm.
\]
Substituting this travelling wave ansatz into the system \eqref{eq:main_system_dissipative} and denoting $\xi = \frac{x - vt}{\varepsilon}$, we get the system
\begin{equation*}
\begin{cases} 
-v s_\xi + f(s, c)_\xi = s_{\xi\xi}, \\
-v (cs + a(c))_\xi + (cf(s,c))_\xi  = (c s_\xi)_\xi + c_{\xi\xi}.
\end{cases}
\end{equation*}
Integrating the equations over $\xi$ we arrive at the travelling wave dynamic system
\begin{equation}\label{eq:dyn_sys_cap_diff}
\begin{cases} 
s_\xi = f(s, c) - v (s + d_1), \\
c_\xi = v (d_1 c - d_2 - a(c)).
\end{cases}
\end{equation}
The values of $d_1$ and $d_2$ are obtained from the boundary conditions:
\begin{align*}
    vd_1 & = -vs^\pm + f(s^\pm, c^\pm), \\
    vd_2 & = v d_1 c^\pm - v a(c^\pm),
\end{align*}
namely, in the case when $c^+ \neq c^-$,
\begin{equation*}
d_1 = \dfrac{a(c^-) - a(c^+)}{c^- - c^+}, \quad d_2 =  \dfrac{c^+ a(c^-) - c^- a(c^+)}{c^- - c^+}.
\end{equation*}
Additionally, the same boundary conditions yield us the Rankine--Hugoniot conditions
\begin{equation}
\label{eq:RH-1}
\begin{split}
    v[s]&=[f(s,c)],
    \\
    v[cs+a(c)]&=[cf(s,c)],
\end{split}
\end{equation}
where $[q(s,c)]=q(s^+,c^+)-q(s^-,c^-)$\footnote{Note the order of ``$+$'' and ``$-$'' terms in this definition. It could be different in different sources. We follow certain proof schemes of \cite{Serre1}, so our order coincides with their.}. Thus, for every set of shock parameters $(s^{\pm}, c^{\pm})$ and $v$ satisfying \eqref{eq:RH-1}, we can construct a phase portrait for the dynamic system \eqref{eq:dyn_sys_cap_diff}. The points $(s^\pm, c^\pm)$ are critical for this dynamic system due to \eqref{eq:RH-1}, and we can check if there is a trajectory connecting the corresponding critical points. But even just analyzing the geometric meaning of the Rankine--Hugoniot conditions \eqref{eq:RH-1}, we derive a lot of restrictions on admissible shock parameters.

\begin{proposition}\label{prop:inadmissible_shocks}
The following restrictions on admissibility are evident from the properties (F1)--(F4), (A1)--(A3), the Rankine--Hugoniot conditions \eqref{eq:RH-1} and the analysis of the sign of the right-hand side of \eqref{eq:dyn_sys_cap_diff}:
\begin{itemize}
    \item Admissible shock velocity $v$ is bounded and strictly positive: $0 < v < \|f\|_{\mathcal C^1}$.
    \item Shocks with $s^- = 0$ cannot be admissible.
    \item Shocks with $s^+ = s^-$ cannot be admissible.
    \item Shocks with $c^+ > c^-$ cannot be admissible.
    \item If $s^+ = 0$ then $c^+=c^-$.
\end{itemize}
\end{proposition}

\begin{lemma}\label{lemma:Lax_for_small_s}
There exists $s_* \in (0,1)$ such that when $s^- < s_*$, we have
\begin{equation}\label{eq:Lax_for_small_s}
f_s(s^-, c^-) > v
\end{equation}
for all admissible shock parameters.
\end{lemma}
\begin{proof}
If $c^- \neq c^+$, we rewrite \eqref{eq:RH-1} in the following form:
\[
    v = \dfrac{[f(s,c)]}{[s]} = \dfrac{f(s^\pm,c^\pm)}{s^\pm + h}, \qquad h = \dfrac{[a(c)]}{[c]},
\]
therefore, points $(-h, 0)$, $(s^+, f(s^+,c^+))$, $(s^-, f(s^-,c^-))$ are collinear and lie on the line $l(s)=v(s+h)$. When $0 \leqslant v \leqslant (1+h)^{-1}$, there is a unique point of intersection of $l(s)$ and $f(s, c^-)$ inside the interval $(0,1)$ due to (F1)--(F3), which we denote $(s^c(c^-, v), f(s^c(c^-, v), c^-))$, and at this point of intersection we have 
\[
f_s(s^c(c^-, v), c^-) > v.
\]
\begin{figure}[htbp]
    \centering
    \includegraphics[width=0.45\textwidth]{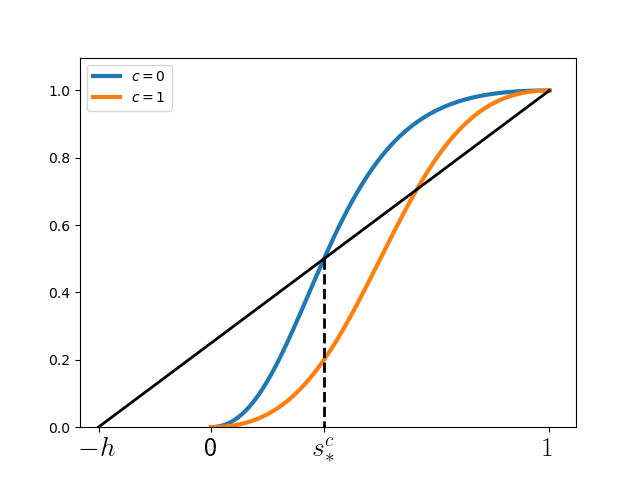}
    \includegraphics[width=0.45\textwidth]{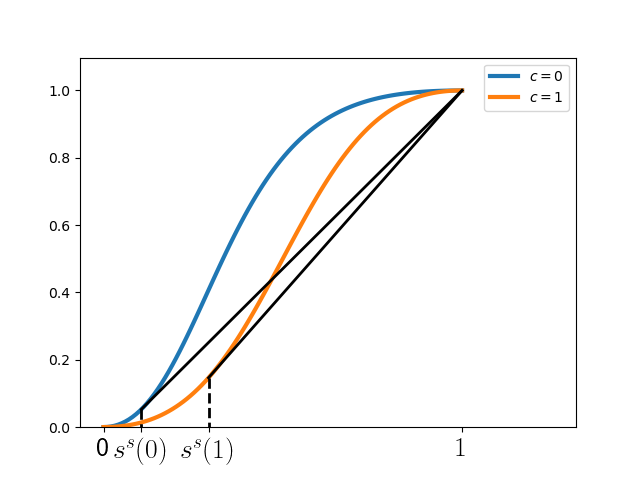}\\
        (a)\qquad\quad\hfil\qquad \qquad
    (b)\hfil
    \caption{The construction of (a) estimate $s_*^c$, (b) function $s^s(c)$. }
    \label{fig:sc_ss_construction}
\end{figure}
Therefore, \eqref{eq:Lax_for_small_s} holds for admissible shocks with $c^- \neq c^+$ for any $s^- < s_*^c$, where
\[
s_*^c = \min\limits_{c\in[0,1]} s^c(c, (1+h)^{-1}).
\]
It is clear that $s^c(c, v)$ is monotone in $c$, therefore
\[
s_*^c =
\begin{cases}
s^c(0, (1+h)^{-1}), &\text{ in the case (F4.1),}\\
s^c(1, (1+h)^{-1}), &\text{ in the case (F4.2),}
\end{cases}
\]
therefore it is obviously separated from $0$ (see Fig.~\ref{fig:sc_ss_construction}a).

If $c^-=c^+=c$, we look at the system \eqref{eq:dyn_sys_cap_diff} and note that it simplifies into one equation. We give a more detailed analysis of this case in Sect.~\ref{sec:sec2-s-shock-admissibility}. For this proof we only need to note that for $s^+ > s^-$ due to \eqref{Oleinik_admissibility} we need the graph of $f(\cdot, c)$ to be above the chord connecting $(s^-, f(s^-, c))$ and $(s^+, f(s^+, c))$. Therefore, denoting by $s^s(c)$ the lower end of the straight portion of the convex hull constructed below the graph of $f(\cdot, c)$ (the maximum of the graphs that are convex and below $f(\cdot, c)$, see Fig.~\ref{fig:sc_ss_construction}b) 
and by
\[
s_*^s = \min\limits_{c\in[0,1]} s^s(c),
\]
we conclude that for admissible shocks with $c^-=c^+$ for any $s^- < s_*^s$ we always have $s^+ < s^-$. Thus, \eqref{eq:Lax_for_small_s} holds due to $f$ being convex at $s^-$. It is possible to construct $f(s,c)$ in such a way that $s^s(c)$ is not monotone in $c$, but by construction it is a positive continuous function on $[0,1]$ (see Fig.~\ref{fig:sc_ss_construction}b), so it admits a minimum separated from zero.

Finally, denoting $s_* = \min\{ s_*^s, s_*^c \} > 0$, we complete the proof.
\end{proof}

\subsection{Nullcline configuration classification}
\label{sec:sec2-nullclines}

In this section we provide the results of a simplified version of the analysis done in \cite[Sect.~4.2]{Bahetal}. In the domain $\{ s\in(0,1); c\in(c^+, c^-) \}$ we draw nullclines $f(s, c) - v (s + d_1) = 0$ as black lines (see 
Fig.~\ref{fig:phase_portrait_all}).

Denote by $s_{1,2}^\pm(v)$ the solutions of $f(s, c^\pm) - v (s + d_1) = 0$ on the $c^\pm$-boundaries of $\Omega$. Since $f$ is S-shaped due to (F3), we know that for $d_1 > 0$ there exist at most $2$ solutions on each side. We assume $s_1^\pm(v) < s_2^\pm(v)$ when both solutions exist. When only one solution exists on the side, we denote it $s_1^+(v)$ for the $c^+$ side or $s_1^-(v)$ for the $c^-$ side. Finally, we denote by $u_{1,2}^\pm(v) = (s_{1,2}^\pm(v), c^\pm)$ the respective points on the boundary of the domain.

Since (F4) only allows functions $f$ that depend monotonically on $c$, there are far fewer types of nullcline configurations for the dynamical system \eqref{eq:dyn_sys_cap_diff} than there were in \cite{Bahetal}. They are as follows (see Fig.~\ref{fig:phase_portrait_all}):
\begin{itemize}
    \item Type 0. One nullcline (black curve) connecting $u^+_1(v)$ to $u^-_1(v)$.
    \item Type I. Two non-intersecting monotone curves connecting the critical points: one from $u^+_1(v)$ to $u^-_1(v)$ and another from $u^+_2(v)$ to $u^-_2(v)$.
    \item Type II. One black curve separated from one side and connecting critical points on the other side.
\end{itemize}

\begin{figure}[H]
    \centering
    \begin{minipage}{0.16\textwidth}
    \includegraphics[width=\linewidth]
    {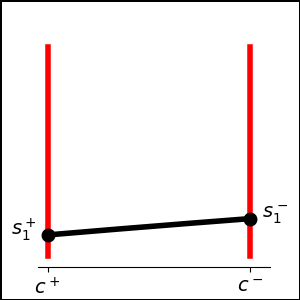}
    \captionof*{figure}{Type 0}
    \end{minipage}
    \hfil
    \begin{minipage}{0.16\textwidth}
    \includegraphics[width=\linewidth]{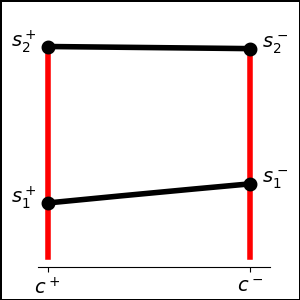}
    \captionof*{figure}{Type 0-I}
    \end{minipage}
    \hfil
    \begin{minipage}{0.16\textwidth}
    \includegraphics[width=\linewidth]{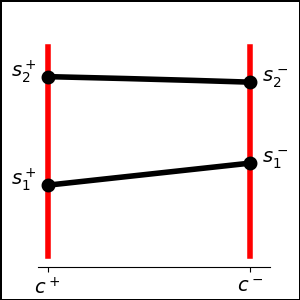}
    \captionof*{figure}{Type I}
    \end{minipage}
    \hfil
    \begin{minipage}{0.16\textwidth}
    \includegraphics[width=\linewidth]{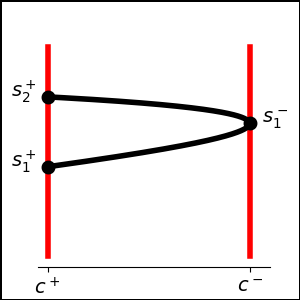}
    \captionof*{figure}{Type I-II}
    \end{minipage}
    \hfil
    \begin{minipage}{0.16\textwidth}
    \includegraphics[width=\linewidth]{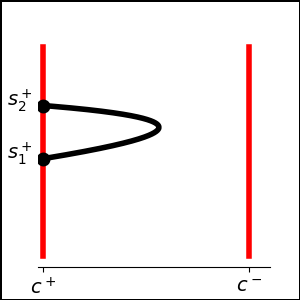}
    \captionof*{figure}{Type II}
    \end{minipage}
    \caption{Three wide classes and two intermediate types of nullcline configurations for the case (F4.1) in the order of increasing $v$.}
    \label{fig:phase_portrait_all}
\end{figure}

In addition, there are singular (occurring for a single value of $v$) intermediate types of nullcline configurations that are essentially border cases for the wide classes described above (see Fig.~\ref{fig:phase_portrait_all}): 
\begin{itemize}
    \item Type 0-I. Similar to Type I, but the upper black curve coincides with the border $s=1$.
    \item Type I-II. Two critical points on one side converge into one. Thus $c^-$ (in the case (F4.1)) or $c^+$ (in the case (F4.2)) have only one critical point on the corresponding boundary, i.e. only $s^-_1(v)$ or $s^+_1(v)$ respectively. 
\end{itemize}

\begin{figure}[htbp]
    \centering
    \includegraphics[width=0.55\textwidth]{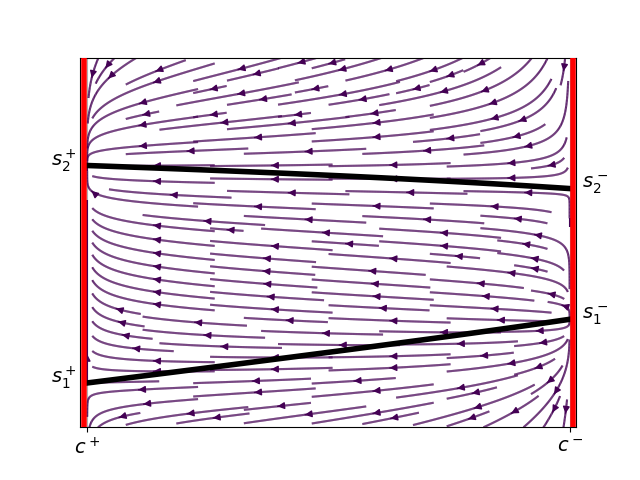}
    \caption{An example of a Type I phase portrait constructed from the system \eqref{eq:dyn_sys_cap_diff}}
    \label{fig:Type-I-phase-portrait}
\end{figure}

Based on this classification it is easy to trace the possible and impossible trajectories (see e.g. Fig.~\ref{fig:Type-I-phase-portrait} for Type I trajectories) and name admissible and inadmissible $c$-shocks (shocks with $c^+\neq c^-$).
\begin{proposition}
\label{prop:c-shock-admissibility}
The local vanishing viscosity condition (W4) imposes the following restrictions on the $c$-shocks:
\begin{itemize}
    \item As already noted in Proposition \ref{prop:inadmissible_shocks}, no trajectories could exist from $c^-$ to $c^+$ if $c^- < c^+$, therefore such shocks are inadmissible.
    \item Travelling waves connecting $u_2^-$ to $u_1^+$ in configurations of Type I and Type 0-I are inadmissible, since there is no trajectory connecting those critical points.
    \item All other travelling waves between critical points $u^-_{1,2}$ and $u^+_{1,2}$ are admissible.
\end{itemize}
\end{proposition}

\section{Lagrange coordinate transformation}
\label{sec:Lagrange}

\subsection{Zero flow area}
\label{sec:zero-flow}
We would like to utilize the Lagrange coordinates, i.e. the coordinates tied to the flow, but to use them rigorously, we first need to establish the area, where there is no flow. Due to (F1) and (F2), the flow function is zero only when $s=0$. In this section we prove that under the conditions (S1)--(S3) it is a connected area bounded by the ray $(x^0, +\infty)$ on one side and by discontinuities of the function $s$ on the other.
\begin{lemma}\label{lemma_zero_below}
Let $(x_*, t_*)$, $x_*>0$, $t_*>0$ be a zero of the solution $s$, i.e. $s(x_*, t_*) = 0$ in the smoothness area or one of $s(x_*\pm0, t_*) = 0$ on the shock. Then $s(x_*, t) = 0$ for all $0 \leqslant t < t_*$.
\end{lemma}
\begin{proof}

Denote by $\Omega_s$ the area, where $s$ and $c$ are $\mathcal{C}^1$-smooth, and by $\{ \Gamma_i \}$ the set of all shocks defined by curves $\gamma^i(t)$. By $E_r$ we denote the set of zero points that are each a limit point at the right end of $\gamma_i$ but are not themselves the point of $\gamma_i$:
\[
(x, t) \in E_r \Longleftrightarrow \exists i : x = \lim\limits_{\tau\to t-0} \gamma^i(\tau), \text{ but } \lim\limits_{\tau\to t-0} \gamma^i_t(\tau) = 0, \text{ so } (x,t)\notin \Gamma_i.
\]
Clearly, 
\[
\partial \Omega_s = \Gamma_x \cup \Gamma_t \cup \Big( \bigcup\limits_i \overline{\Gamma}_i  \Big),
\] 
where $\Gamma_x = \{(x,t):t=0, x\geqslant 0\}$ is the $x$ axis, $\Gamma_t = \{(x,t): x=0, t\geqslant 0 \}$ is the $t$ axis and $\overline{\Gamma}_i$ is the closure of the curve $\Gamma_i$.
Since $x_* > 0, t_* > 0$, there are 3 possible cases for $(x_*, t_*)$:
\begin{itemize}
    \item[Case 1.] $(x_*, t_*) \in \inter \Omega_s$ is in the interior of the smoothness area.
    \item[Case 2.] $(x_*, t_*) \in \overline{\Gamma}_i\setminus E_r$ is on the boundary $\partial\Omega_s$, but not in $E_r$.
    \item[Case 3.] $(x_*, t_*) \in E_r$.
\end{itemize}
We will resolve each case separately in order.

\underline{Case 1.} In the compact neighbourhood $K\subset \inter \Omega_s$ of $(x_*, t_*)\in K$ the derivatives of $s$ and $c$ are bounded. Therefore, there exists a unique solution of the characteristic equation
\[
x_t(t) = f_s(s(x(t),t), c(x(t),t)), \quad x(t_*) = x_*,
\]
since the right-hand side of this equation is locally Lipschitz, and thus, the Picard--Lindel\"{o}f theorem applies. Along this characteristic due to the first equation of the system \eqref{eq:main_system_chem_flood} we have
\begin{align*}
\dfrac{d}{dt} s(x(t), t) &= s_t(x(t), t) + x_t(t) s_x(x(t), t) \\
{} & = s_t(x(t), t) + f_s(s(x(t), t), c(x(t), t)) s_x(x(t), t) \\
{} & = -f_c(s(x(t), t), c(x(t), t)) c_x(x(t), t).
\end{align*}
Due to (F1) we have an estimate
\[
|f_c(s,c)| \leqslant Ms,
\]
where $M = \|f\|_{\mathcal C^2}$. Therefore, along the characteristic we can estimate
\[
-M|c_x| \leqslant \dfrac{\frac{d}{dt}s}{s} \leqslant M|c_x|.
\]
Integrating this estimate over $(t_1, t_2)$ we obtain
\begin{equation}\label{eq:characteristic_s_estimate}
s(x(t_2), t_2) \mathrm{e}^{-MC_K(t_2-t_1)} \leqslant s(x(t_1), t_1) \leqslant s(x(t_2), t_2) \mathrm{e}^{MC_K(t_2-t_1)},
\end{equation}
where $C_K$ is given in (W2) in Definition \ref{def:solution}. Substituting $t_2 = t_*$ we obtain $s(x(t_1), t_1) = 0$ for all points along the characteristic and therefore $x_t = f_s(s,c) = 0$, thus the characteristic is a vertical line carrying the value $s=0$. We can extend the characteristic until we leave $\inter\Omega_s$, at which point we either reach $\Gamma_x$, which finishes the proof, or we reach $\overline{\Gamma}_i$ and switch to Case 2 or Case 3. Moreover, the characteristic cannot reach $\Gamma_i$, since $s^-=s(\gamma_i(t)-0, t)=0$ is not admissible due to Proposition \ref{prop:inadmissible_shocks}, so it can only reach $\overline{\Gamma}_i\setminus \Gamma_i$, thus remaining in $\Omega_s$.

\underline{Case 2.} Similarly, due to Proposition \ref{prop:inadmissible_shocks} it is impossible for $(x_*, t_*)\in \Gamma_i$ to have $s(x_*-0, t_*)=0$ above the shock. Therefore, we have $s(x_*+0, t_*)=0$. First, we choose $\varepsilon > 0$ in such a way that (see Fig.~\ref{fig:case2-illustration}):
\begin{itemize}
    \item Inside the half-ball $B^-_\varepsilon = B_\varepsilon(x_*, t_*)\cap \{(x,t): x\leqslant x_* \}$ for some $\delta>0$ all shocks have the velocity $v > \delta$ at every point.
    \item There are no points of $\partial \Omega_s$ in $$B_\varepsilon^{\wideangleup} = B_\varepsilon(x_*, t_*)\cap \{(x,t): x \geqslant x_* - (t_*-t)\delta, t \leqslant t_* \}\setminus \{(x_*, t_*)\}.$$
\end{itemize}
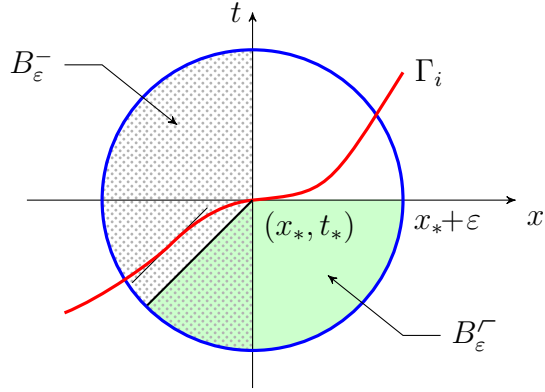
\begin{figure}
    \begin{minipage}{\linewidth}
    \begin{center}
    \begin{tikzpicture}[>=stealth', yscale=1, xscale=1]
\filldraw [green!20] (0, 0) -- (2, 0)  arc (0:-134.5:2) --cycle;
\draw [pattern=crosshatch dots, pattern color = black!30] (0, 2) arc (90:270:2) --cycle;
\draw[thin] (-1.1-0.2*2.5, -0.6 - 0.2*2.5) -- (-1.1+0.2*2.5, -0.6 + 0.2*2.5);
\draw[->, thin] (-3,0) -- (3.5,0)
node[below right] {$x$}; 
\draw[->, thin] (0,-2.5) -- (0,2.5)
node[left] {$t$};
\draw[very thick, blue] (0,0) circle (2);
\node[below right] at (0,0) {$(x_*, t_*)$};
\node[below right] at (2,0) {$x_*\!\!+\!\varepsilon$};
\draw[very thick, red] (-2.5, -1.5) ..
controls (-2.5 +0.5, -1.5 + 0.2) and (-1.1-0.2, -0.6 - 0.2)  .. (-1.1, -0.6) ..
controls (-1.1+0.2, -0.6 + 0.2) and (-1*0.5, -0.1*0.5) .. (0, 0) ..
controls (1, 0.1) .. (2, 1.7) node[right, black] {$\Gamma_i$};
\draw[thick] (0, 0) -- (-0.2*7, -0.2*7);
\draw[->] (-2.5, 1.8) node[left] {$B^-_\varepsilon$} -- (-2, 1.8) -- (-1, 1);  
\draw[->] (2.5, -1.8) node[right] {$B^{\wideangleup}_\varepsilon$} -- (2, -1.8) -- (1, -1);
\end{tikzpicture}
    \end{center}
    \end{minipage}
    \caption{The neighbourhood of the point $(x_*, t_*)$, the set $B_\varepsilon^{\wideangleup}$ (light green area) and the set $B_\varepsilon^-$ (dotted area).}
    \label{fig:case2-illustration}
\end{figure}
Indeed, let's begin with the first requirement. Suppose there is no such $\varepsilon$. Then there exists a sequence of shock points $(x_k, t_k)$ approaching $(x_*, t_*)$ from the left with shock velocities $v_k \to 0$. Since there is a finite number of shocks nearby, we can select a subsequence belonging to just one shock $\Gamma$ given by $\gamma(t)$. Then 
\[
x_* = \lim\limits_{t\to t_*-0} \gamma(t), \text{ and } \lim\limits_{t\to t_*-0} \gamma_t(t) = 0,
\]
so $(x_*, t_*) \notin \Gamma$, since shocks with velocity $0$ cannot be admissible. Hence, $(x_*, t_*)\in E_r$, which is prohibited in Case 2. Therefore, it is possible to choose $\varepsilon>0$ to satisfy the first requirement. 

For the second requirement, we note that there is a finite number of shocks in $B_\varepsilon^{\wideangleup}$ and none of them approach $(x_*, t_*)$, since any shock approaching $(x_*, t_*)$ has the velocity $v > \delta$ to the left of it and $v < +\infty$ to the right, therefore it cannot go into $B_\varepsilon^{\wideangleup}$ due to its definition. Therefore, it is possible to choose $\varepsilon$ small enough to exclude all of the shocks and their limiting points from $B_\varepsilon^{\wideangleup}$.

Now, $B_\varepsilon^{\wideangleup} \subset \inter\Omega_s$, therefore it is possible to construct a characteristic from any point in it. We will look at characteristics from every point $(r, t_*)$ as $r\to x_*+0$: 
\[
x_t(t) = f_s(s(x(t),t), c(x(t),t)), \quad x(t_*) = r,
\]
Due to an estimate similar to \eqref{eq:characteristic_s_estimate} with $K = \overline{B}_\varepsilon^{\wideangleup}$ we conclude that $s$ on these characteristics uniformly tends to $0$ in $B_\varepsilon^{\wideangleup}$ as $r\to x_*+0$. Therefore, denoting $$R_r = \big[0, s(r,t_*)\mathrm{e}^{\varepsilon MC_{K}}\big]\times[0,1]$$ we conclude that for $(x(t),t)\in B_\varepsilon^{\wideangleup}$ on the characteristic we can estimate
\[
f_s(s(x(t),t),c(x(t),t)) < \max\limits_{(s,c)\in R_r} f_s(s , c) \to 0 \quad \text{ as } r \to x_*+0,
\]
so $x_t$ on these characteristics uniformly tends to zero. When $x$ is sufficiently close to $x_*$ we obtain an estimate $x_t < \delta/2$, so in a certain neighbourhood of $x_*$ the characteristics stay inside $B_\varepsilon^{\wideangleup}$ until they reach its circular boundary.  Thus, the characteristics uniformly approach the vertical line, bringing with them values of $s$ tending to $0$. Since we are inside $\Omega_s$, we conclude that $s(x_*, t)=0$ for all $0 < t_* - t < \varepsilon$. After that we can switch to Case 1.

\underline{Case 3.} Suppose $\Gamma$ is the shock given by the curve $\gamma(t)$, for which 
\[
x_* = \lim\limits_{t\to t_*-0} \gamma(t), \quad \lim\limits_{t\to t_*-0} \gamma_t(t) = 0.
\]
Due to the Ranikine-Hugoniot condition \eqref{eq:RH-1} and (F2) we know that when the velocity of the shock tends to $0$, and the values of $s$ on one side of the shock tend to $0$, the values of $s$ on both sides of the shock tend to $0$:
\[
\lim\limits_{t\to t_*-0} s(\gamma(t)+0, t) = \lim\limits_{t\to t_*-0} s(\gamma(t)-0, t) = 0.
\]
There is a finite number of shocks nearby, so we can choose $\varepsilon >0 $ in such a way that $\Gamma$ does not intersect other shocks inside $B_\varepsilon^-$. Denote by $(x_1, t_1) = \Gamma \cap \partial B_\varepsilon^-$ the point where $\Gamma$ leaves $B_\varepsilon^-$ and by
\[
\gamma^{\max}(t) := \max \big(\{ \gamma^i(t): \gamma^i(t) < \gamma(t) \} \cup \{ 0 \}\big), \quad t \in (t_1, t_*).
\]
\begin{figure}
    \begin{minipage}{\linewidth}
    \begin{center}
    \def\lowborder{-8}
\def\leftborder{-2}
\def\upperborder{3}
\def\rightborder{12}
\def\startx{9}
\def\starty{0}
\def\radius{7}
\def\smalleps{0.2}
\def\tOne{-6.494}
\begin{tikzpicture}[>=stealth', yscale=0.5, xscale=0.5]
\definecolor{mycolor}{rgb}{0.2,0.6,0.3}
	\begin{scope}[xshift=4pt, color=mycolor]
		\draw[in=-90, line width = 4pt, opacity = 0.6] (1-0.15, -6-0.1) parabola (4+0.1, -4+0.1) -- (0, -4+0.1) -- (0, -2.8-0.15) -- (3.5, -2.8-0.15) -- (4.5-0.1, -2-0.1) to (\startx, \starty);
	\end{scope}
\draw[thin,->] ({\leftborder}, 0) -- (\rightborder, 0) node[below right] {$x$};
\draw[thin,->] (0,\lowborder) -- (0,\upperborder) node[left] {$t$};
\draw[very thick] (1, -6) parabola (4, -4);
\draw[in=-90, very thick] (4.5, -2) to (\startx, \starty);
\draw[very thick] (3.5, -2.8) -- (6.5, -0.4);
\draw[very thick] (1.5, -1.7) arc (-80:-30:7);
\node[below left] at (0, 0) {$t_*$};
\node[below right] at (\startx, \starty) {$x_*$};
\draw[thick, blue] (\startx, \starty - \radius) arc (-90:-200:\radius);
\draw[thick, blue] (\startx, \starty - \radius) arc (-90:-70:\radius);
\draw[in=-90, out=20, very thick] (4.5, -7.5) to (\startx, \starty);
\draw[in=-90, out=20, very thick, color=red] (2, -7) to (\startx, \starty);
\node[right] at (7.5, -5.5) {$\gamma(t)$};
\node[above, color=red] at (6, -5) {$\widetilde{\gamma}(t)$};
\draw[thick, dashed] (0, \tOne) -- (6.45 , \tOne);
\node[left] at (0, \tOne) {$t_1$};

\end{tikzpicture}
    \end{center}
    \end{minipage}
    \caption{An illustration for the construction of $\widetilde{\gamma}(t)$. Black lines are shocks. Green highlight denotes the path of $\gamma^{\max}(t)$.}
    \label{fig:case3-illustration}
\end{figure}
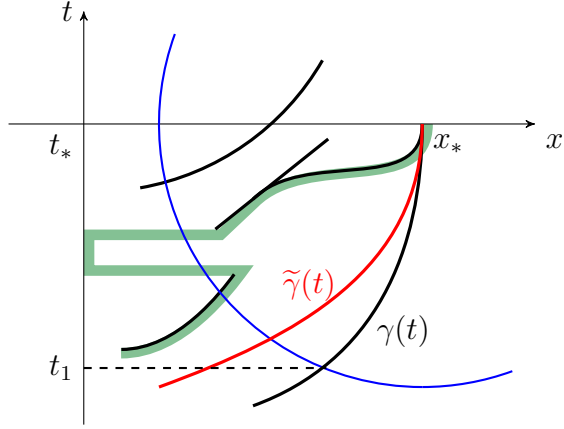
Between $\gamma(t)$ and $\gamma^{\max}(t)$ inside $B^-_\varepsilon$ we select an arbitrary curve $\widetilde{\gamma}(t)$ (see Fig.~\ref{fig:case3-illustration}) such that above $t_{1/2}=\frac{t_*+t_1}{2}$ we have
\begin{align*}
&(\widetilde{\gamma}(t), t) \in B^-_\varepsilon, &\quad& t\in(t_{1/2}, t_*),\\
&\gamma^{\max}(t) < \widetilde{\gamma}(t) < \gamma(t), &\quad& t\in(t_{1/2}, t_*), \\
& \lim\limits_{t\to t_*-0} \widetilde{\gamma}(t) = x_*.
\end{align*}
Note that by construction $\{(x,t): t_{1/2}<t<t_*, \widetilde{\gamma}(t) \leqslant x < \gamma(t) \} \subset \Omega_s$. Now we can construct a characteristic from every point $(\widetilde{\gamma}(t),t)$. Then as $t\to t_*-0$ the characteristics will uniformly tend to vertical line, thus after some point, all of the characteristics must intersect $\Gamma$. At the point of intersection we obtain a contradiction with Lemma \ref{lemma:Lax_for_small_s}  for all $t$, such that $s(\widetilde{\gamma}(t),t)\mathrm{e}^{\varepsilon MC_K} < s_*$, where $K=\overline{B}_\varepsilon^-$, and $s_*$ is given in Lemma \ref{lemma:Lax_for_small_s}. The characteristic with incline $f_s(s,c)$ intersects the shock while going above it, therefore at the point of intersection we have
\[
f_s(s(x-0, t),c(x-0,t)) \leqslant \gamma_t(t), \quad s(x-0, t) \leqslant s(\widetilde{\gamma}(t),t)\mathrm{e}^{\varepsilon MC_K} < s_*,
\]
which contradicts Lemma \ref{lemma:Lax_for_small_s}. Therefore, we conclude that Case 3 is impossible.
\end{proof}

\begin{remark}
We demonstrated in section \ref{sec:sec2-Hopf} above that the classical Rankine--Hugoniot condition \eqref{eq:RH-1} follows from local vanishing viscosity admissibility (W4). Therefore, similar to \cite[Lemma 2.2.1]{Serre1}, any solution $(s,c)$ of the system \eqref{eq:main_system_chem_flood}  (see Definition \ref{def:solution})   
with initial and boundary conditions \eqref{eq:Initial_boundary_problem} satisfies in $Q = \mathbb{R}_+\times \mathbb{R}_+$ the integral equations
\begin{equation}
\label{eq:weak-s-eq}
\iint\limits_Q s \psi_t + f(s,c)\psi_x \,dx\,dt + 
\int\limits_{\mathbb{R}_+} s_0^x(x) \psi(x, 0) \, dx +
\int\limits_{\mathbb{R}_+} f(s_0^t(t), c_0^t(t)) \psi(0, t) \, dt = 0
\end{equation}
and
\begin{align}
\label{eq:weak-c-eq}
\begin{split}
\iint\limits_Q (sc + a) \psi_t + cf(s,c)\psi_x \,dx\,dt &+ \int\limits_{\mathbb{R}_+} (s_0^x(x)c_0^x(x) + a(c_0^x(x)) \psi(x, 0) \, dx \\ 
& + \int\limits_{\mathbb{R}_+} c_0^t(t)f(s_0^t(t), c_0^t(t)) \psi(0, t) \, dt
 = 0,
\end{split}
\end{align}
where $\psi\in\mathcal D(Q)$ is an arbitrary test function. Here and below by $\mathcal D(\cdot)$ we denote the space of smooth functions with compact supports that are allowed to have non-zero values on the boundary of the domain.
\end{remark}

\begin{proposition}\label{prop:contour_integral}
For any solution $(s,c)$ the differential form $f(s,c)\,dt-s\,dx$ derived from the first equation of \eqref{eq:main_system_chem_flood} is exact, i.e. on any closed curve $\partial \Omega$ with finite number of shock points we have
\begin{equation}\label{eq:contour_int}
\oint\limits_{\partial \Omega} f(s,c)\,dt - s\, dx = 0.
\end{equation}
Similarly, from the second equation of \eqref{eq:main_system_chem_flood} we derive the exact form $(cs + a(c)) dx - f(s,c) dt$, therefore
\begin{equation}\label{eq:contour_int2}
    \oint\limits_{\partial\Omega} c (s\, dx - f(s,c)\, dt) + \oint\limits_{\partial\Omega} a(c)\, dx = 0.
\end{equation}
\end{proposition}
\begin{proof}
Note that (W3), that is the fact that $(s, c)$ are the classical solutions of the system \eqref{eq:main_system_chem_flood} inside $\Omega_s$, holds if and only if the differential forms $f(s,c)\,dt - s\, dx$ and $(sc + a(c))\,dx - cf(s, c)\,dt$ are closed in any simply connected open subset of $\Omega_s$ due to the Poincar\'{e} lemma. Therefore, \eqref{eq:contour_int} and \eqref{eq:contour_int2} are clearly fulfilled for any $\Omega \subset \Omega_s$.

For an arbitrary $\Omega$ consider the integral equality \eqref{eq:weak-s-eq}. By approximating the characteristic function of the set $\Omega$ with smooth test functions, and then going to the limit, we arrive at \eqref{eq:contour_int}. The relation \eqref{eq:contour_int2} is obtained similarly form \eqref{eq:weak-c-eq}.

\end{proof}

\begin{lemma}
\label{lemma:t0_is_shocks}
For all $x > x^0$ we define
\[
t_0(x) = \sup\{ t: s(x, t) = 0 \}.
\]
Then 
\begin{itemize}
    \item $t_0(x) < +\infty$;
    \item $(x, t_0(x))$ is a point on a shock;
    \item $t_0(x)$ is continuous, piece-wise $\mathcal{C}^1$-smooth.
\end{itemize}
\end{lemma}
\begin{proof}
\underline{Step 1.} Suppose there exists $x_1 > x^0$ such that $t_0(x_1) = +\infty$. Therefore, $s(x_1, t) = 0$ for all $t \geqslant 0$. Consider a rectangle $R_{\tau} = [0,x_1]\times [0, \tau]$. Since we do not have admissible shocks with velocity $0$ or $\infty$, the boundaries of $R_\tau$ have a finite number of shock points. Therefore, by Proposition \ref{prop:contour_integral} we have
\begin{align*}
0 &= \oint\limits_{\partial R_\tau} f(s,c)\,dt - s\,dx \\
{} &= \int\limits_0^{x_1} s(x, \tau)\,dx - \int\limits_0^{x_1} s_0^x(x)\,dx + \int\limits_0^{\tau} f(s(x_1, t), c(x_1, t))\,dt - \int\limits_0^{\tau} f(s_0^t(t), c_0^t(t))\,dt.
\end{align*}
For these integrals we have estimates
\begin{align*}
& \int\limits_0^{x_1} s_0^x(x)\,dx \geqslant 0, \qquad \int\limits_0^{x_1} s(x, \tau)\,dx \leqslant x_1, \\
& \int\limits_0^{\tau} f(s(x_1, t), c(x_1, t))\,dt = \int\limits_0^{\tau} f(0, c(x_1, t))\,dt = 0, \\
& \int\limits_0^{\tau} f(s_0^t(t), c_0^t(t))\,dt \geqslant \tau \varepsilon, \quad \varepsilon = \min\big\{ f(s,c): s\in[\delta^0, 1], c\in[0,1] \big\},
\end{align*}
where $\delta^0$ is defined in (S3). Therefore, we obtain
\[
0 = \oint\limits_{\partial R_\tau} f(s,c)\,dt - s\,dx \leqslant x_1 - \tau\varepsilon,
\]
which clearly breaks for $\tau > x_1/\varepsilon$, which leads to a contradiction. Thus, $t_0(x) = +\infty$ is impossible, which proves the first assertion of this lemma.

\underline{Step 2.} Suppose there is $x_1 > x^0$, such that $Y=(x_1, t_0(x_1))$ is not a shock point. Let's consider all possible cases for $Y$:
\begin{itemize}
    \item $Y\in \inter \Omega_s$. Then similar to Case 1 of Lemma \ref{lemma_zero_below} we can construct a characteristic through it, that turns out to be a vertical line carrying the value $s=0$. If we continue this characteristic upwards, we conclude that there are values $t>t_0(x_1)$ for which $s(x_1, t)=0$, therefore arriving at a contradiction. Thus, $Y\notin \inter \Omega_s$.
    \item $Y\in E_r$. Similarly, Case 3 of Lemma \ref{lemma_zero_below} prohibits this case as well.
    \item $Y\in E_l$, where $E_l$ is similar to $E_r$ but consists of the left ends of shocks:
    \[
    (x, t) \in E_l \Longleftrightarrow \exists i: x = \lim\limits_{\tau\to t+0} \gamma^i(\tau), \text{ but } \lim\limits_{\tau\to t+0} \gamma^i_t(\tau) = 0, \text{ so } (x,t)\notin \Gamma_i.
    \]
    Since $Y$ is not a point of any shock, it is a point of continuity of $s$.
    Therefore, $s(x_1\pm 0, t_0(x_1)) = 0$, and thus $x_1 \neq x^0$, since $s(x^0-0, 0) \geqslant \delta^0 > 0$ due to (S2). Since the set $E_l$ is locally finite, we can assume $(x_1, t_0(x_1))$ is the first such point, i.e. $(x, t_0(x)) \in \Gamma_j$ for some $j$ for all $x\in(x^0, x_1)$.  Now let us consider the limit
    \[
    t^* = \lim_{x\to x_1-0} t_0(x).
    \]
    It is clear that either $(x_1, t^*) \in \Gamma_j$ for some $j$, or $(x_1, t^*) \in E_r$. Let's go over possible relations between $t^*$ and $t_0(x_1)$:
    \begin{itemize}
        \item If $t^* > t_0(x_1)$, then it is clear that $s(x_1, t) = 0$ for all $t<t^*$, which contradicts the definition of $t_0(x_1)$.
        \item If $t^* \leqslant t_0(x_1)$, then we see that $s(x_1, t^*+0) = s(x_1, t^*-0) = 0$, therefore, due to Proposition \ref{prop:inadmissible_shocks} $(x_1, t^*)$ cannot be a point of an admissible shock, therefore $(x_1, t^*) \in E_r$, which also leads to a contradiction with Case 3 of Lemma \ref{lemma_zero_below}.
    \end{itemize}
    Therefore, we conclude, that $Y\in E_l$ is impossible as well.
\end{itemize}
After eliminating every other possibility we arrive at $Y\in\Gamma_i$ for some $i$, which proves the second assertion of the lemma.

The third assertion follows immediately. Any discontinuity in $t_0(x)$ will create either a point in $E_r$, or a point in $E_l$, but both of these were just eliminated as impossible. Therefore, $t_0(x)$ is continuous and consists of a locally finite number of $\mathcal C^1$-smooth pieces.
\end{proof}

\begin{corollary}\label{corollary:zero_flow_area}
Define $\Omega_0 = \{(x,t): x > x^0, 0\leqslant t < t_0(x)\}$. Then $s(x,t) = 0$ in $\Omega_0$ and $s(x,t) > 0$ outside $\overline{\Omega}_0$. Moreover, $s(x,t)$ is locally separated from $0$ outside $\overline{\Omega}_0$.
\end{corollary}

\begin{proposition}\label{prop:c_on_zero_boundary}
$c_t = 0$ in $\Omega_0$, therefore $c(x, t_0(x)) = c^x_0(x)$.
\end{proposition}

\subsection{Lagrange coordinates}
\label{sec2-Lagrange}
It is hard to trace back the history of the coordinates transformation described in this subsection. Many authors describe the transformation with no citations of previous work. The oldest reference we found is \cite{Courant} cited in \cite{Wa87} in the context of gas dynamics equations. The idea is also presented in the lectures by Gelfand \cite{Gelfand} for the case of an arbitrary system of conservation laws. The splitting technique for the system \eqref{eq:main_system_chem_flood} using the Lagrange coordinate transformation is presented in \cite{PiBeSh06}. It is later developed and applied to different systems by many authors (see \cite{Pires2021} and references therein).

We denote by $\varphi$ the potential such that
\begin{equation} 
\label{eq:dPhi}
    d\varphi=f(s,c)\,dt-s\,dx.
\end{equation}
To explain the physical meaning of $\varphi$ let us consider any trajectory $\nu$ connecting $(0, 0)$ and $(x, t)$.
When $s$ denotes the saturation of some liquid, the potential $\varphi(x, t)$ is equal to the amount of this liquid
passing through the trajectory:
\begin{equation}
\label{eq:phi}
    \varphi(x,t)=\int\limits_{\nu} f(s,c)\,dt-s\,dx.
\end{equation}
This coordinate change is only applicable in the area $Q_{orig} = Q\setminus \overline{\Omega}_0$, where the saturation $s$ and the flow function $f(s,c)$ are not zero. It keeps the $x$ coordinate, so it maps the axis $\Gamma_t$ onto itself. 
The segment $[0, x^0]\times\{0\}$ maps into a curve $( \varphi_0(x), x)$, where
\begin{equation}
\label{eq:def_varphi_0}
\varphi_0(x) = -\int\limits_0^x s^x_0(r) \, dr.
\end{equation}
The curve $(x, t_0(x))$ for $x > x^0$ maps into a horizontal line beginning at the point $(\varphi_0(x^0), x^0)$. Therefore $Q_{orig}$ maps into 
\[
Q_{lagr} = Q \cup \{ (x, \varphi) : 0 < x \leqslant x^0, 0 > \varphi > \varphi_0(x) \} \cup \left((x^0, +\infty) \times (\varphi_0(x^0), 0)\right),
\]
see Fig.~\ref{fig:orig-lagr-areas}.

Corollary \ref{corollary:zero_flow_area} guaranties that there is a reverse transform given by
\[
dt=\frac1{f(s,c)}\,d\varphi+\frac s{f(s,c)}\,dx,
\]
and the denominators are locally separated from zero, therefore this coordinate change is a piecewise $\mathcal C^1$-diffeomorphism.
Since $\mathcal C^1$-smooth curves preserve their smoothness properties under any diffeomorphism, all discontinuity curves map into $\mathcal C^1$-smooth discontinuity curves. 

Substituting \eqref{eq:dPhi} into \eqref{eq:contour_int2}, we get that the form $  - c \,d \varphi + a(c) \,dx$ is also exact in any area in $Q_{lagr}$ where the images of $(s, c)$ are $\mathcal C^1$-smooth. Therefore, it is closed too and leads to the identity
\[
0 = d(c \,d \varphi - a(c)\, dx) = \left( \frac{\partial c}{\partial x} + \frac{\partial a(c)}{\partial \varphi}\right)\, dx \wedge d\varphi.
\]
Together with the identity
\[
0 = d(dt) = d\left( \dfrac{1}{f} \, d\varphi + \dfrac{s}{f} \, dx\right) = \left(\frac{\partial}{\partial x}\left(\frac1f\right) - \frac{\partial}{\partial \varphi}\left(\frac s f\right)\right) \, dx \wedge d\varphi
\]
it gives us inside the areas of $\mathcal C^1$-smoothness the classical system
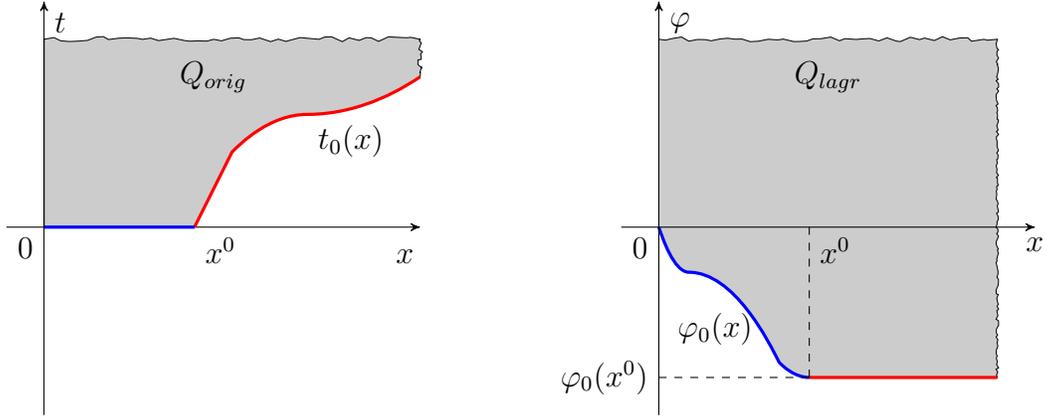
\begin{figure}
    \begin{minipage}{\linewidth}
    \begin{center}
    \def\lowborder{-5}
\def\leftborder{-1}
\def\upperborder{6}
\def\rightborder{10}
\def\startpoint{\rightborder * 0.4}
\def\phinote{0.8 * \lowborder}
\def\eps{1}
\begin{tikzpicture}[>=stealth', yscale=0.5, xscale=0.5]
\draw[thin,->] ({\leftborder}, 0) -- (\rightborder, 0) node[below] {$x\phantom{x^0}$};
\begin{scope}
\filldraw[fill=black!20, ultra thin]
(\startpoint, 0)  -- (\startpoint+1, 2) parabola[bend at end] (\startpoint + 3, 3) parabola (\startpoint + 6, 4)
  decorate [decoration={random steps,segment length=1pt,amplitude=1pt}] {-- (\startpoint + 6, \upperborder - \eps)}
 decorate [decoration={random steps,segment length=3pt,amplitude=1pt}] {--(0, \upperborder - \eps)}
 -- (0, 0) -- cycle;
\draw[very thick, red] (\startpoint, 0)  -- (\startpoint+1, 2) parabola[bend at end] (\startpoint + 3, 3) parabola (\startpoint + 6, 4);
\node[below right] at (\startpoint + 3, 3) {$t_0(x)$} ;
\node[] at (\rightborder * 0.45, 4) {$Q_{orig}$};
\node[below right] at (\startpoint, 0) {$x^0$};
\node[below left] at (0, 0) {$0$};
\draw[very thick, blue] (0, 0) -- (\startpoint, 0);
\draw[thin,->] (0,\lowborder) -- (0,\upperborder) node[below right ] {$t$};
\end{scope}
\end{tikzpicture}
\hspace{1cm}
\begin{tikzpicture}[>=stealth', yscale=0.5, xscale=0.5]
\begin{scope}
\filldraw[fill=black!20, ultra thin]
(0, 0)  parabola[bend at end] (\startpoint * 0.2, \phinote * 0.3)
parabola (\startpoint * 0.8, \phinote *0.9) parabola[bend at end] (\startpoint, \phinote) -- (\rightborder - \eps, \phinote)
  decorate [decoration={random steps,segment length=1pt,amplitude=0.5pt}] {-- (\rightborder - \eps, \upperborder - \eps)}
 decorate [decoration={random steps,segment length=3pt,amplitude=1pt}] {--(0, \upperborder - \eps)}
 -- (0, 0) -- cycle;
 \draw[thin,->] ({\leftborder}, 0) -- (\rightborder, 0) node[below] {$x$};
\draw[very thick, blue] (0, 0)  parabola[bend at end] (\startpoint * 0.2, \phinote * 0.3)
parabola (\startpoint * 0.8, \phinote *0.9) parabola[bend at end] (\startpoint, \phinote);
\node[] at (\rightborder * 0.45, 4) {$Q_{lagr}$};
\node[below right] at (\startpoint, 0) {$x^0$};
\node[below left] at (0, 0) {$0$};
\draw[very thick, red] (\startpoint, \phinote) -- (\rightborder - \eps, \phinote);
\draw[dashed]  (\startpoint, \phinote) -- (\startpoint, 0);
\draw[dashed] (\startpoint, \phinote) -- (0, \phinote);
\node[left] at (0, \phinote) {$\varphi_0(x^0  )$};
\node[below] at (\startpoint/2-0.5, \phinote/2) {$\varphi_0(x)$};
\end{scope}
\draw[thin,->] (0,\lowborder) -- (0,\upperborder) node[below right ] {$\varphi$};
\end{tikzpicture}
    \end{center}
    \end{minipage}
    \caption{Areas $Q_{orig}$ (grey area on the left), and $Q_{lagr}$ (grey area on the right). The red curve on the left is mapped onto the red line on the right. The blue line on the left is mapped onto the blue curve on the right.}
    \label{fig:orig-lagr-areas}
\end{figure}
\begin{align*}
\begin{split}
    \frac{\partial}{\partial x}\left(\frac1f\right) - \frac{\partial}{\partial \varphi}\left(\frac s f\right)&=0, \\
    \frac{\partial c}{\partial x} + \frac{\partial a(c)}{\partial \varphi}&=0. 
\end{split}
\end{align*}
We use the notation
\begin{equation} 
\label{eq:def-U-F}
\mathcal{U}(\varphi, x)=\frac1{f(s(x,t),c(x,t))}, \quad \zeta(\varphi,x) = c(x,t), \quad\text{and}\quad\mathcal F(\mathcal U,\zeta)=-\frac{s}{f(s,c)}
\end{equation}
in order to transform this system into the system of conversation laws
\begin{align}
       \mathcal U_x + \mathcal F(\mathcal U,\zeta)_\varphi & = 0,\label{eq:U-lagr-eqn}\\
    \zeta_x + a(\zeta)_\varphi&=0. \label{eq:c-lagr-eqn}
\end{align}

\begin{proposition}
Suppose that $(s,c)$ is a solution of the problem \eqref{eq:main_system_chem_flood}  (see Definition \ref{def:solution})
with initial and boundary conditions \eqref{eq:Initial_boundary_problem} satisfying the conditions (S1)--(S3). Then the functions $(\mathcal U, \zeta)$ given by the formulae \eqref{eq:def-U-F} satisfy the integral equations
\begin{align}
\label{eq:U_weak}
\begin{split}
    \iint\limits_{Q_{lagr}} \mathcal U \widetilde \psi_x + \mathcal F(\mathcal U, \zeta) \widetilde \psi_\varphi  & \,d\varphi\,dx +
    \int\limits_0^\infty\mathcal U_0^\varphi(\varphi) \widetilde \psi(\varphi, 0)\,d\varphi \\ 
    & + \int\limits_{x^0}^\infty
 \mathcal F(\mathcal U(\varphi_0(x^0)+0, x), c_0^x(x))\widetilde \psi(\varphi_0(x^0), x)\,dx = 0
\end{split}
\end{align}
and
\begin{align}
\label{eq:zeta_weak}
\begin{split}
    \iint\limits_{Q_{lagr}} \zeta\widetilde\psi_x + a(\zeta) \widetilde\psi_\varphi & \,d\varphi\,dx + 
    \int\limits_0^{x^0}  (s_0^x(x)c_0^x(x) + a(s_0^x(x))) \widetilde\psi(\varphi_0(x), x)\,dx \\
    {} & + \int\limits_{x^0}^\infty  a(c_0^x(x)) \widetilde\psi(\varphi_0(x^0), x)\,dx +
    \int\limits_{0}^\infty \zeta_0^\varphi(\varphi)\widetilde\psi(\varphi, 0) \, d\varphi  = 0
\end{split}
\end{align}
for all $\widetilde \psi \in \mathcal D(Q_{lagr})$,
where the initial values $\mathcal U_0^\varphi$ and $\zeta_0^\varphi$ are given by 
\begin{equation}
\label{eq:U_zeta_weak_initial_values}
\mathcal U_0^\varphi(\varphi(0, t)) = \dfrac{1}{f(s_0^t(t), c_0^t(t))}, \qquad \zeta_0^\varphi (\varphi(0, t)) = c_0^t(t).
\end{equation}
\end{proposition}
\begin{proof}
Every test function $\widetilde \psi \in \mathcal D(Q_{lagr})$ corresponds to a continuous piecewise $\mathcal C^1$-smooth function $\psi$ with compact support in $Q_{orig}$ given by
$$\psi(x, t) = \widetilde\psi(\varphi(x, t), x).$$
The derivatives of $\psi$ satisfy the following relations due to \eqref{eq:dPhi}:
\begin{equation}
\label{eq:tilde_psi_derivatives}
    \psi_t = \widetilde\psi_\varphi f,
    \qquad
    \psi_x= \widetilde\psi_x - s\widetilde\psi_\varphi.
\end{equation}
Let's prove \eqref{eq:zeta_weak} first, since it is a more straightforward substitution. Note, that the integral equality \eqref{eq:weak-c-eq} holds for $\psi$, since we can approximate it with smooth functions. Substituting \eqref{eq:tilde_psi_derivatives} into the main term of \eqref{eq:weak-c-eq} and taking into account the fact that the Jacobian of the change of variables $(x, t) \to (\varphi(x, t), x)$ is equal to $f$, we get
\begin{equation}
\label{eq:zeta_weak_eq_1}
    \iint\limits_{Q_{orig}} (sc + a(c)) \psi_t + cf(s,c)\psi_x \,dx\,dt = 
    \iint\limits_{Q_{lagr}} a(\zeta) \widetilde\psi_\varphi + \zeta\widetilde\psi_x \,d\varphi\,dt .
\end{equation}
Since by definition
$\psi(x, 0) = \widetilde\psi(\varphi_0(x), x)$ for $x \leqslant x^0$ and $\psi(x, 0) = \widetilde\psi(\varphi_0(x^0), x)$ for $x > x^0$,
we have
\begin{align}
\label{eq:zeta_weak_eq_2}
\begin{split}
    \int\limits_0^{x^0} (s_0^x(x)c_0^x(x) & + a(s_0^x(x))) \psi(x, 0) \, dx \\
    & = \int\limits_0^{x^0}  (s_0^x(x)c_0^x(x) + a(s_0^x(x))) \widetilde\psi(\varphi_0(x), x)\,dx,
\end{split}
\end{align}
\begin{equation}
\label{eq:zeta_weak_eq_3}
    \int\limits_{x^0}^\infty (s_0^x(x)c_0^x(x) + a(c_0^x(x))) \psi(x, 0) \, dx = 
    \int\limits_{x^0}^\infty  a(c_0^x(x)) \widetilde\psi(\varphi_0(x^0), x)\,dx.
\end{equation}
We have $s_0^x(x)=0$ in \eqref{eq:zeta_weak_eq_3} due to (S1).
Introducing the initial values for $\zeta$ given by \eqref{eq:U_zeta_weak_initial_values} and using the relation
\begin{equation}
\label{eq:varphi_0_t}
\varphi(0, t) = \int\limits_0^t f (s_0^t(\tau), c_0^t(\tau))\, d \tau,
\end{equation}
derived from \eqref{eq:phi}, we obtain
\begin{equation}
\label{eq:zeta_weak_eq_4}
    \int\limits_0^\infty c_0^t(t)f(s_0^t(t), c_0^t(t)) \psi(0, t) \, dt = 
    \int\limits_0^\infty \zeta_0^\varphi(\varphi)\widetilde\psi(\varphi, 0) \, d\varphi.
\end{equation}
Adding up the equations \eqref{eq:zeta_weak_eq_1}, \eqref{eq:zeta_weak_eq_2}, \eqref{eq:zeta_weak_eq_3} and \eqref{eq:zeta_weak_eq_4}, in view of \eqref{eq:weak-c-eq} we derive \eqref{eq:zeta_weak}.

To obtain \eqref{eq:U_weak}, we consider the Stokes identity
\begin{equation}
\label{eq:U_weak_eq_1}
\iint\limits_{Q_{orig}}\psi_x \,dt \,dx = \int\limits_{\partial Q_{orig}}\!\!\!\!\psi n_x \,d\sigma.
\end{equation}
We derive the following simple identity by adding and subtracting the same term:
\[
\psi_x = \left(\frac1{f}\left(\psi_x + \frac{s}{f}\psi_t\right) - \frac{s}{f}\frac1f\psi_t\right) f.
\]
With this identity, applying \eqref{eq:def-U-F} and \eqref{eq:tilde_psi_derivatives}, we transform the left-hand side of \eqref{eq:U_weak_eq_1} in terms of $\mathcal U$ and $\zeta$ as follows
\begin{equation}
\label{eq:U_weak_main_term}
\iint\limits_{Q_{orig}}\psi_x \,dt \,dx =
\iint\limits_{Q_{lagr}} \mathcal U \widetilde \psi_x + \mathcal F(\mathcal U, \zeta) \widetilde \psi_\varphi \,d\varphi\,dx.
\end{equation}
Since on the boundary of $Q_{orig}$ corresponding to $t_0(x)$ we have
\[
n_x = \dfrac{t'_0(x)}{\sqrt{1 + (t'_0(x))^2}} \quad\text{and}\quad
d\sigma = \sqrt{1 + (t'_0(x))^2} \, dx,
\]
the right-hand side of \eqref{eq:U_weak_eq_1} could be clarified as
\[
\int\limits_{\partial Q_{orig}}\!\!\!\!\psi n_x \,d\sigma = -\int\limits_0^\infty\psi(0, t)\,dt + \int\limits_{x_0}^\infty \psi(x, t_0(x)) t'_0(x)\,dx.
\]
For the first term, using \eqref{eq:varphi_0_t} and the definition \eqref{eq:U_zeta_weak_initial_values} of $\mathcal U_0^\varphi$ we derive
\[
\int\limits_0^\infty\psi(0, t)\,dt = \int\limits_0^\infty\mathcal U_0^\varphi(\varphi) \widetilde \psi(\varphi, 0)\,d\varphi.
\]
For the second term, since $t_0(x)$ is composed of shocks with $s(x+0, t_0(x)) = 0$ due to Lemma \ref{lemma:t0_is_shocks}, a straightforward computation of the Rankine--Hugoniot condition \eqref{eq:RH-1} shows that
$$t'_0(x) = \frac{s(x-0, t_0(x))}{f(s(x-0, t_0(x)), c(x, t_0(x)))},$$
and using \eqref{eq:def-U-F} and Proposition \ref{prop:c_on_zero_boundary}
we obtain
$$t'_0(x) = -\mathcal F(\mathcal U(\varphi_0(x^0)+0, x), c_0^x(x)).$$
Therefore, we have
\begin{align}
\label{eq:psi_nx}
\begin{split}
\int\limits_{\partial Q_{orig}}\!\!\!\!\psi n_x \,d\sigma = &-\int\limits_0^\infty\mathcal U_0^\varphi(\varphi) \widetilde \psi(0,\varphi)\,d\varphi \\
&- \int\limits_{x_0}^\infty \mathcal F(\mathcal U(\varphi_0(x^0)+0, x),  c_0^x(x)) \widetilde \psi(\varphi_0(x^0), x) \,dx.
\end{split}
\end{align}
Adding up the equations \eqref{eq:U_weak_eq_1}, \eqref{eq:U_weak_main_term} and \eqref{eq:psi_nx} we arrive at \eqref{eq:U_weak}.

Note that \eqref{eq:U_weak} is missing a boundary integral corresponding to the $\varphi_0(x)$ part of the boundary. That is because
\[
\int\limits_0^{x^0} \left(-\mathcal U(\varphi_0(x), x) \varphi'_0(x) + \mathcal F(\mathcal U(\varphi_0(x), x), \zeta(\varphi_0(x), x))\right) \widetilde{\psi}(\varphi_0(x), x) \, dx = 0
\]
due to the definition of the initial-boundary values on this boundary, since due to \eqref{eq:def_varphi_0} and \eqref{eq:def-U-F} for every $x\in[0, x^0]$ we have
\[
\varphi'_0(x) = -s_0^x(x), \quad \mathcal U(\varphi_0(x), x) = \dfrac{1}{f(s_0^x(x), c_0^x(x))},
\]
\[
\mathcal F(\mathcal U(\varphi_0(x), x), \zeta(\varphi_0(x), x)) = \dfrac{-s_0^x(x)}{f(s_0^x(x), c_0^x(x))},
\]
and therefore
\[
-\mathcal U(\varphi_0(x), x) \varphi'_0(x) + \mathcal F(\mathcal U(\varphi_0(x), x), \zeta(\varphi_0(x), x)) = 0.
\]
\end{proof}

\begin{remark}
The same reasoning as in \cite[Lemma 2.2.1]{Serre1} shows that on every shock the equations in the weak form \eqref{eq:U_weak} and \eqref{eq:zeta_weak} result in the Rankine--Hugoniot condition
\begin{equation} 
\label{eq:RH-2}
\begin{split}
    v^*[\mathcal U]&=[\mathcal F(\mathcal U,\zeta)],
    \\
    v^*[\zeta]&=[a(\zeta)],
\end{split}
\end{equation} 
where $v^*$ is the velocity of the shock between states $(\mathcal U^-, \zeta^-)$ and $(\mathcal U^+, \zeta^+)$. Here, like in original coordinates, $[q(\mathcal U, \zeta)]=q(\mathcal U^+, \zeta^+)-q(\mathcal U^-, \zeta^-)$.
\end{remark}

Properties of the new flow function $\mathcal F$ (see Fig.~\ref{fig:BLf_lagr}) that correspond to the properties (F1)--(F4) of the function $f$ are listed below.
\begin{figure}[htbp]
    \centering
    \includegraphics[width=0.55\textwidth]{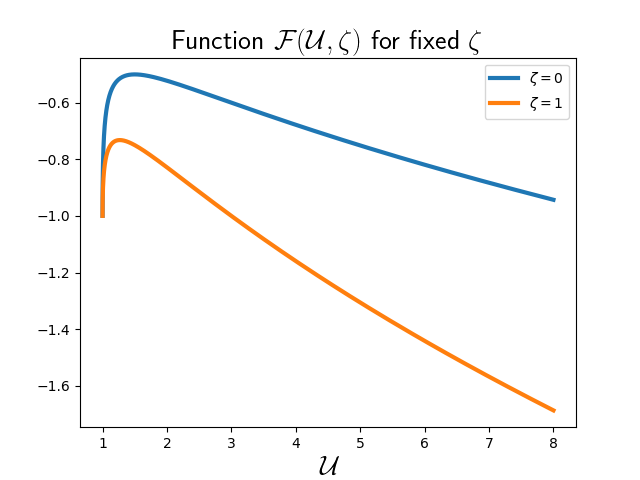}
    \caption{The function $\mathcal F(\mathcal U,\zeta)$ corresponding to the flow function $f(s, c)$ plotted in Fig \ref{fig:BL_ads} (a).}
    \label{fig:BLf_lagr}
\end{figure}
\begin{proposition}
\label{prop:new-flow-function-properties}
For all $\zeta\in[0,1]$ the following properties of the function $\mathcal F$ are fulfilled
\begin{itemize}
\item[($\mathcal F$1)]
\begin{itemize}
\item[$\bullet$]
$\mathcal F \in \mathcal C^2([1,+\infty)\times[0,1])$; 
\item[$\bullet$]
$\mathcal F(\mathcal U,\zeta) < 0$ for all $\,\mathcal U\in [1,+\infty)$;
\item[$\bullet$]
$\mathcal F(1, \zeta) = - 1$;
\item[$\bullet$]
$\lim\limits_{\mathcal U\to\infty}\mathcal{F}(\mathcal U, \zeta) = -\infty$;
\end{itemize}
 \medskip
\item[($\mathcal F$2)]
\begin{itemize}
\item[$\bullet$]
$\lim\limits_{\mathcal U\to1}\mathcal{F}_{\mathcal U}(\mathcal U, \zeta) = +\infty$;
\item[$\bullet$]
$\lim\limits_{\mathcal U\to\infty}\mathcal{F}_\mathcal{U}(\mathcal U, \zeta) = 0$;
\item[$\bullet$]
there exists a unique global maximum $\mathcal U^{\max}(\zeta)$;
\item[$\bullet$]
 $\mathcal F_{\mathcal U}(\mathcal U, \zeta) > 0$ for $\mathcal U \in (1, \mathcal U^{\max}(\zeta))$; 
 \item[$\bullet$]
 $\mathcal F_{\mathcal U}(\mathcal U, \zeta) < 0$ for
 $\mathcal U \in (\mathcal U^{\max}(\zeta), +\infty)$;
 \end{itemize}
 \medskip
 \item[($\mathcal F$3)]
 \begin{itemize}
\item[$\bullet$] 
function $\mathcal F(\cdot,\zeta)$ has a unique point of inflection $\mathcal U^I(\zeta) \in (1, +\infty)$, such that $\mathcal F_{\mathcal U\mathcal U} (\mathcal U, \zeta) < 0$ for $1 < \mathcal U < \mathcal U^I(\zeta)$ and $\mathcal F_{\mathcal U\mathcal U} (\mathcal U, \zeta) > 0$ for $\mathcal U > \mathcal U^I(\zeta)$. \item[$\bullet$] $\mathcal U^{I}(\zeta) > \mathcal U^{\max}(\zeta)$;
\item[$\bullet$]
therefore, if $\mathcal U$, $\zeta$ are such that $\mathcal F_{\mathcal U} (\mathcal U, \zeta) > 0$, then $\mathcal F_{\mathcal U\mathcal U} (\mathcal U, \zeta) < 0$;
\end{itemize}
\item[($\mathcal F$4)]
function 
$\mathcal F$ is monotone in $\zeta$, i.e.  for all $\mathcal U\in(1, +\infty)$ one of the following properties holds true 
\begin{itemize}
\item[($\mathcal F$4.1)]
$\mathcal F_\zeta(\mathcal U, \zeta) < 0$ in the case (F4.1), 
\item[($\mathcal F$4.2)]
$\mathcal F_\zeta(\mathcal U, \zeta) > 0$ in the case (F4.2).
\end{itemize}
\end{itemize}
\end{proposition}
\begin{proof}
To obtain the last identity  of ($\mathcal F$1) we apply L'H\^{o}pital's rule to the definition~\eqref{eq:def-U-F} and use property (F2). The rest of ($\mathcal F$1) immediately follows from the definition \eqref{eq:def-U-F}.  Note that, for a fixed $c=\zeta$ there is a one-to-one correspondence between $s$ and $\mathcal U$, so we can put $\mathcal U = \mathcal U(s)$. 
Due to the chain rule we have
\begin{equation*}
    \mathcal{F}_{\mathcal{U}}(\mathcal{U}(s), \zeta) = \frac{d}{ds}\mathcal F(\mathcal{U}(s), \zeta)\left(\frac{d}{ds}\mathcal U(s)\right)^{-1}.
\end{equation*}
Hence, using \eqref{eq:def-U-F}, we derive
\begin{equation} \label{eq:dF_U}
    \mathcal F_\mathcal U(\mathcal U, \zeta) = \frac{f(s, c)}{f_s(s,c)} - s.
\end{equation}
The function $\psi(s,c) = f(s, c) - s f_s(s, c)$ is zero at $s=0$. Its derivative admits $\psi_s(s,c) = sf_{ss}(s, c)$, which means that due to (F3) the function $\psi(\cdot, c)$ decreases on $(0, s^I(c))$ and increases on $(s^I(c), 1)$. Since $\psi(1, c) = 1$, the function $\psi(\cdot, c)$ has exactly one positive root $s^{wf}(c)\in(s^I(c), 1)$ ($wf$ stands for ``water front'', since this concentration arises at the front of certain Riemann problem solutions).
Due to  \eqref{eq:dF_U} and (F2) the value $\mathcal U^{\max}(\zeta)$ corresponding to $s^{wf}(c)$ is the global maximum point of $\mathcal F(\cdot, \zeta)$ for every $\zeta\in[0,1]$, and the inequalities in ($\mathcal F$2) are fulfilled.
Moreover, we see that for $0 < s< s^I(c)$
\[
f(s,c) < sf_s(s,c),
\]
 and hence,
\[
0< \frac{f(s,c)}{f_s(s,c)} < s.
\] 
Thus, the squeeze theorem yields
\[
\lim\limits_{\mathcal U\to\infty}\mathcal{F}_\mathcal{U}(\mathcal U(s), \zeta) = \lim\limits_{s\to 0}\frac{f(s,c)}{f_s(s,c)} - s = 0.
\]
Since $\mathcal U = 1$ corresponds to $s=1$, 
 we obtain $\lim\limits_{\mathcal U\to1}\mathcal{F}_{\mathcal U}(\mathcal U, \zeta) = +\infty$ from \eqref{eq:dF_U} and (F2).

Again applying the chain rule we derive the following identity
\begin{equation*}
    \mathcal{F}_{\mathcal{UU}}(\mathcal U(s), \zeta) = \frac{d}{ds}\mathcal F_{\mathcal U}(\mathcal U(s), \zeta)\left(\frac{d}{ds}\mathcal U(s)\right)^{-1}.
\end{equation*}
Using \eqref{eq:def-U-F} and \eqref{eq:dF_U}, we obtain
\begin{equation*}
    \mathcal{F}_{\mathcal{UU}}(\mathcal U(s), \zeta) = f_{ss}(s, c)\frac{f^3(s, c)}{f_s^3(s, c)}.
\end{equation*}
Due to (F1) and (F2) the sign of $\mathcal{F}_{\mathcal{UU}}(\mathcal U(s), \zeta)$ coincides with one of  $f_{ss}(s, c)$. According to (F3) function $f$ has a unique inflection point, therefore so does $\mathcal F$.
Since $0<s^I(c)<s^{wf}(c)<1$, the corresponding values in the Lagrange coordinates are in reverse order $1 < \mathcal U^{\max}(\zeta) <\mathcal U^I(\zeta) < +\infty$ and thus 
the inflection cannot happen while $\mathcal F_{\mathcal U}(\mathcal U, \zeta) > 0$ which proves the last assertion in ($\mathcal F$3). Straightforward computations shows  that property (F4) results in ($\mathcal F$4).
\end{proof}

\section{Entropy conditions in Lagrange coordinates}
\label{sec:entropy}

\subsection{Mapping shocks to Lagrange coordinates}
When constructing the solution in Lagrange coordinates, we need to discern which shocks are admissible. One possible approach (see e.g. \cite{Shen}) is to use the vanishing viscosity method directly for the system \eqref{eq:U-lagr-eqn}--\eqref{eq:c-lagr-eqn} to establish admissibility criteria for the shocks. However, we consider that approach flawed. It adds second-order terms into the equation system in Lagrange coordinates, but those terms lack any kind of physical significance. In some simple cases it could be argued that the set of admissible shocks is the same regardless of how we add dissipative terms, but it was shown in \cite{Bahetal} that even for the original system with a more complex dependence of the flux function on $c$ we have different admissible shocks for different ratios of dissipative parameters. Therefore, the best approach is to choose the most physically meaningful form of dissipative terms to tie the admissibility of shocks to something that could be experimentally measured in real systems.

That is why we established admissibility in original coordinates with dissipative system \eqref{eq:main_system_dissipative}. Now, in order to define admissibility in Lagrange coordinates, we use \eqref{eq:def-U-F} to construct a map between shocks in original coordinates and shocks in Lagrange coordinates. Only shocks corresponding to vanishing viscosity admissible shocks in original coordinates will be considered admissible in Lagrange coordinates. This mapping goes as follows. 

Consider a shock in original coordinates at the point $(x_1, t_1)$ with values
\[
s^\pm = s(x_1\pm 0, t_1), \quad c^\pm = c(x_1\pm 0, t_1).
\]
Denote $\theta_c(s) = \frac{1}{f(s,c)}$ and $\vartheta_c = \theta_c^{-1}$ its inverse function with respect to its argument $s$. Using these functions we map
\[
\mathcal U^{[+]} = \theta_{c^+}(s^+), \quad \mathcal U^{[-]} = \theta_{c^-}(s^-), \quad \zeta^{[+]} = c^+, \quad \zeta^{[-]} = c^-,
\]
\[s^+ = \vartheta_{\zeta^{[+]}}(\mathcal U^{[+]}), \quad s^- = \vartheta_{\zeta^{[-]}}(\mathcal U^{[-]}), \quad c^+ = \zeta^{[+]}, \quad c^- = \zeta^{[-]}. 
\]
Note, that in the original coordinates values $s^\pm$ correspond to $x\to x_1\pm 0$ respectively. The shock velocity in original coordinates is always positive due to Proposition \ref{prop:inadmissible_shocks}, therefore, $s^\pm$ correspond to $t\to t_1\mp 0$:
\[
s^\pm = s(x_1, t_1 \mp 0), \quad c^\pm = c(x_1, t_1\mp 0).
\]
Further, due to \eqref{eq:dPhi} and (F1), when $x = x_1$ is fixed, $t\to t_1\mp 0$ correspond to $\varphi\to \varphi_1\mp 0$ for the point $(\varphi_1, x_1)$ on a corresponding shock in Lagrange coordinates, and so do $\mathcal U^{[\pm]}$:
\[
\mathcal U^{[\pm]} = \mathcal U(\varphi_1 \mp 0, x_1), \quad \zeta^{[\pm]} = \zeta(\varphi_1 \mp 0, x_1).
\]
But for the equations \eqref{eq:U-lagr-eqn}, \eqref{eq:c-lagr-eqn} in Lagrange coordinates the $x$ axis plays the role of time and $\varphi$ the role of space, so we would like to denote $\mathcal U^\pm$ to correspond to $\varphi\to \varphi_1\pm 0$. Thus, we denote
\[
\mathcal U^+ = \mathcal U^{[-]}, \quad \mathcal U^- = \mathcal U^{[+]}, \quad \zeta^+ = \zeta^{[-]}, \quad \zeta^- = \zeta^{[+]},
\]
\[
\mathcal U^{\pm} = \mathcal U(\varphi_1 \pm 0, x_1), \quad \zeta^{\pm} = \zeta(\varphi_1 \pm 0, x_1),
\]
and obtain a one-to-one mapping of shocks in original and Lagrange coordinates:
\begin{equation}\label{shock_mapping}
(\mathcal U^\pm, \zeta^\pm) \to (\vartheta_{\zeta^\mp}(\mathcal U^\mp), \zeta^\mp), \quad (s^\pm, c^\pm) \to (\theta_{c^\mp}(s^\mp), c^\mp).
\end{equation}

\subsection{Oleinik, Lax and entropy admissibility for $\mathcal U$-shocks}
\label{sec:sec2-s-shock-admissibility}
Consider an $s$-shock in original coordinates, i.e. a shock with no change in $c = c^+ = c^-$. The system \eqref{eq:dyn_sys_cap_diff} for this case simplifies into one equation
\[
s_\xi = f(s, c) - f(s^-, c) - v (s - s^-) =: \Psi(s).
\]
For this equation $s^\pm$ are critical points due to the Rankine--Hugoniot condition \eqref{eq:RH-1}. And the existence of a trajectory connecting these critical points depends on the sign of the right-hand side of the equation between the critical points. For the trajectory to exist it is necessary and sufficient that
\begin{equation}\label{Oleinik_admissibility}
\Psi(s)(s^+-s^-) \geqslant 0 \quad \text{ for all } s \text{ between } s^+ \text{ and } s^-.
\end{equation}
This condition is called Oleinik's entropy condition or E-condition \cite{Oleinik}. Note that since $f$ is $S$-shaped, it is impossible for $\Psi(s)$ to be $0$ inside the interval without changing sign, therefore, the sign in this condition could be interpreted as strict without loss of generality. Note also that $\Psi_s(s) = f_s(s, c) - v$. Therefore, for the current problem this criterion is equivalent to the following pair of conditions:
\begin{itemize}
    \item $\Psi(s) \neq 0$ for all $s$ between $s^+$ and $s^-$;
    \item $f_s(s^+, c) \leqslant v \leqslant f_s(s^-, c)$, but both signs cannot be equal at the same time.
\end{itemize}
The inequalities in the second condition of the pair are known as Lax's shock condition.

Now, we'd like to transfer these conditions to Lagrange coordinates, obtaining similar inequalities for $\mathcal F(\mathcal U, \zeta)$, $\zeta = c$ and
\[
\Psi^*(\mathcal U) = \mathcal F(\mathcal U, \zeta) - \mathcal F(\mathcal U^-, \zeta) - v^*(\mathcal U - \mathcal U^-),
\]
\[
v^* = 
\begin{cases}
\dfrac{\mathcal F(\mathcal U^-, \zeta) - \mathcal F(\mathcal U^+, \zeta)}{\mathcal U^- - \mathcal U^+}, & \mathcal U^- < +\infty,\\
0, & \mathcal U^- = +\infty.
\end{cases}
\]
Similar to the original coordinates, $\mathcal U^\pm$ are zeroes of $\Psi^*$ due to the exact formula for $v^*$ obtained from Rankine--Hugoniot conditions \eqref{eq:RH-2}.

\begin{lemma}
\label{lemma_Oleinik_equivalence}
Lax condition for an $s$-shock in original coordinates is equivalent to Lax condition for the corresponding $\mathcal U$-shock in Lagrange coordinates:
\begin{equation}\label{eq:Lax-Lagrange}
\mathcal F_{\mathcal U}(\mathcal U^+, \zeta) \leqslant v^* \leqslant \mathcal F_{\mathcal U}(\mathcal U^-, \zeta),
\end{equation}
and both signs can only be equal at the same time when $\mathcal U^- = +\infty$.

Moreover, $\Psi^*(\mathcal U) \neq 0$ if and only if for the corresponding $s$ we have $\Psi(s) \neq 0$, therefore the Oleinik E-condition is also equivalent for $s$-shocks in original coordinates and $\mathcal U$-shocks in Lagrange coordinates.
\end{lemma}
\begin{proof}
First, we consider the case $s^+ \neq 0$. Let's prove the second inequality, the proof of the first one is similar. Assume we have 
\begin{equation}\label{eq:half-inequality-orig}
f_s(s^+, c) \leqslant v = \dfrac{f(s^-, c)-f(s^+, c)}{s^--s^+}.
\end{equation}
Using \eqref{eq:def-U-F} and \eqref{eq:dF_U} we substitute $s^\pm$, $f$ and $f_s$ into this inequality:
\[
\dfrac{1}{\mathcal F_{\mathcal U}(\mathcal U^-, \zeta)\mathcal U^-  - \mathcal F(\mathcal U^-, \zeta)} \leqslant \dfrac{\mathcal U^- - \mathcal U^+}{\mathcal F(\mathcal U^-, \zeta) \mathcal U^+ - \mathcal F(\mathcal U^+, \zeta)\mathcal U^-}.
\]
Through simple transformations we arrive at
\begin{equation}\label{eq:half-inequality-lagr}
\mathcal F_{\mathcal U}(\mathcal U^-, \zeta) \geqslant \dfrac{\mathcal F(\mathcal U^-, \zeta) - \mathcal F(\mathcal U^+, \zeta)}{\mathcal U^- - \mathcal U^+}.
\end{equation}
This transformations and substitutions work in reverse, therefore the inequalities \eqref{eq:half-inequality-orig} in original and \eqref{eq:half-inequality-lagr} in Lagrange coordinates are equivalent. Also, they can only become equalities simultaneously, therefore inequalities in \eqref{eq:Lax-Lagrange} cannot both be equalities when $\mathcal U^- \neq +\infty$. 

Similarly, if we have $\Psi(s) = 0$, we rewrite it as
\[
\dfrac{f(s, c)-f(s^-, c)}{s-s^-} = \dfrac{f(s^+, c)-f(s^-, c)}{s^+-s^-}.
\]
Substituting $s$, $s^\pm$ and $f$ in this relation we arrive at
\[
\dfrac{\mathcal U^+ - \mathcal U}{\mathcal F(\mathcal U^+, \zeta) \mathcal U - \mathcal F(\mathcal U, \zeta)\mathcal U^+} = \dfrac{\mathcal U^+ - \mathcal U^-}{\mathcal F(\mathcal U^-, \zeta) \mathcal U^+ - \mathcal F(\mathcal U^+, \zeta)\mathcal U^-}.
\]
This transforms into
\[
\dfrac{\mathcal F(\mathcal U, \zeta) - \mathcal F(\mathcal U^+, \zeta)}{\mathcal U - \mathcal U^+} = \dfrac{\mathcal F(\mathcal U^-, \zeta) - \mathcal F(\mathcal U^+, \zeta)}{\mathcal U^- - \mathcal U^+} = v^*,
\]
therefore, the points $(\mathcal U, \mathcal F(\mathcal U, \zeta))$, $(\mathcal U^+, \mathcal F(\mathcal U^+, \zeta))$ and $(\mathcal U^-, \mathcal F(\mathcal U^-, \zeta))$ are collinear, thus we also have
\[
\dfrac{\mathcal F(\mathcal U, \zeta) - \mathcal F(\mathcal U^-, \zeta)}{\mathcal U - \mathcal U^-} = \dfrac{\mathcal F(\mathcal U^-, \zeta) - \mathcal F(\mathcal U^+, \zeta)}{\mathcal U^- - \mathcal U^+} = v^*,
\]
which means that $\Psi(s)\neq 0$ and $\Psi^*(\mathcal U) \neq 0$ are equivalent.

Now we consider the case $s^+ = 0$, and therefore $\mathcal U^- = +\infty$. 
In this case $v^* = 0$ and due to ($\mathcal F$2) in Proposition \ref{prop:new-flow-function-properties} the derivative $\mathcal F_{\mathcal U}(\mathcal U^-, \zeta) = 0 = v^*$ as well, therefore the second inequality holds. On the other hand we have
\[
f_s(s^-, c) \geqslant v = \dfrac{f(s^-, c)}{s^-},
\]
therefore
\[
\dfrac{1}{\mathcal F_{\mathcal U}(\mathcal U^+, \zeta)\mathcal U^+  - \mathcal F(\mathcal U^+, \zeta)} \geqslant \dfrac{-1}{\mathcal F(\mathcal U^+, \zeta) }.
\]
Transforming this relation we obtain
$\mathcal F_{\mathcal U}(\mathcal U^+, \zeta) \leqslant 0 = v^*$.
Therefore, the first inequality holds, and also due to ($\mathcal F$2) in Proposition \ref{prop:new-flow-function-properties} we have $\mathcal F_{\mathcal U}(\mathcal U, \zeta) < 0$ for all $\mathcal U > \mathcal U^+$, thus $\Psi^*(\mathcal U) \neq 0$ holds for all $\mathcal U > \mathcal U^+$. Note that in this case it is possible to have $\mathcal F_{\mathcal U}(\mathcal U^+, \zeta) = \mathcal F_{\mathcal U}(\mathcal U^-, \zeta) = v^* = 0$, but only when $\mathcal U^+ = \mathcal U^{\max}$.
\end{proof}

\begin{proposition}
\label{prop:boundary_shock_oleinik}
Due to Lemma \ref{lemma_Oleinik_equivalence} on the boundary shock $\{(\varphi_0(x^0), x)| x\geqslant x^0\}$ the value above the shock is $\mathcal U^+(\varphi_0(x^0), x) = \mathcal U(\varphi_0(x^0)+0, x)$, therefore we have
\[
\mathcal F_{\mathcal U}(\mathcal U(\varphi_0(x^0)+0, x), \zeta(\varphi_0(x^0), x)) \leqslant 0, \qquad \forall x\geqslant 0.
\]
\end{proposition}

Following the structure of \cite[Section 2.3]{Serre1} in reverse, we prove that Oleinik's E-condition implies the entropy condition for any convex positive entropy, but in particular for entropy--entropy-flux pairs
$(|\mathcal U - k|, \mathcal G(\mathcal U, k))$, $k\in \mathbb{R}$, where
\[
\mathcal G(\mathcal U, k) = (\mathcal F(\mathcal U, \zeta) - \mathcal F(k, \zeta))\sign(\mathcal U - k).
\]
\begin{lemma}\label{lemma-U-entropy-ineq}
On any admissible $\mathcal U$-shock $\Phi(x)$ inside $Q_{lagr}$ we have
\begin{equation}
\label{entropy_inequality}    
[\mathcal G(\mathcal U, k)] \leqslant \dfrac{d\Phi}{dx} [|\mathcal U - k|], \quad \forall k\in\mathbb{R}.
\end{equation}
\end{lemma}
\begin{proof}
Due to Lemma \ref{lemma_Oleinik_equivalence}, similar to \eqref{Oleinik_admissibility} we have
\[
\Psi^*(\mathcal U) (\mathcal U^+ - \mathcal U^-) \geqslant 0.
\]
Using the fact, that $\mathcal U^\pm$ are zeroes of $\Psi^*$, this inequality could be rewritten the following way:
\[
\Big(\gamma\mathcal F(\mathcal U^-, \zeta) + (1-\gamma)\mathcal F(\mathcal U^+, \zeta) - \mathcal F(\gamma\mathcal U^-+(1-\gamma)\mathcal U^+, \zeta)\Big)\sign(\mathcal U^+ - \mathcal U^-) \leqslant 0
\]
for all $\gamma\in(0,1)$. Therefore, for $k = \gamma\mathcal U^-+(1-\gamma)\mathcal U^+$, noting that $$\mathcal U^+ + \mathcal U^- - 2k = (2\gamma - 1)(\mathcal U^+ - \mathcal U^-) = (2\gamma - 1) [\mathcal U]$$ we have
\begin{align*}
\Big( 
&\mathcal F(\mathcal U^-, \zeta) + \mathcal F(\mathcal U^+, \zeta) - 2 \mathcal F(k, \zeta) \\
{} &- \dfrac{[\mathcal F(\mathcal U, \zeta)]}{[\mathcal U]} (\mathcal U^+ + \mathcal U^- - 2k) \Big) \sign(\mathcal U^+ - \mathcal U^-) \leqslant 0,
\end{align*}
and this relation easily rewrites into \eqref{entropy_inequality} for all $k$ between $\mathcal U^+$ and $\mathcal U^-$ due to Rankine--Hugoniot condition (see \eqref{eq:RH-2})
\begin{equation}
\label{Lagrange_Rankine_Hugoniot}
\dfrac{d\Phi}{dx} = \dfrac{[\mathcal F(\mathcal U, \zeta)]}{[\mathcal U]}.    
\end{equation}
Indeed, we just need to note that since $k$ is between $\mathcal U^\pm$, we have
\[
\sign(\mathcal U^+ - \mathcal U^-) = \sign(\mathcal U^+ - k) = -\sign(\mathcal U^- - k), 
\]
therefore
\begin{align*}
[\mathcal G(\mathcal U, k)] &= (\mathcal F(\mathcal U^+, \zeta) - \mathcal F(k, \zeta))\sign(\mathcal U^+ - k) - (\mathcal F(\mathcal U^-, \zeta) - \mathcal F(k, \zeta))\sign(\mathcal U^- - k) \\
& = \Big(\mathcal F(\mathcal U^+, \zeta) + \mathcal F(\mathcal U^-, \zeta) - 2\mathcal F(k, \zeta)\Big)\sign(\mathcal U^+ - \mathcal U^-), \\
[|\mathcal U - k|] &= (\mathcal U^+ - k)\sign(\mathcal U^+ - k) - (\mathcal U^- - k)\sign(\mathcal U^- - k) \\
&= (\mathcal U^+ + \mathcal U^- - 2k) \sign(\mathcal U^+ - \mathcal U^-).
\end{align*}

On the other hand, when $k > \max(\mathcal U^+, \mathcal U^-)$ or $k < \min(\mathcal U^+, \mathcal U^-)$, we have $\sign(\mathcal U^+-k) = \sign(\mathcal U^- - k)$, therefore, after resolving all signs in \eqref{entropy_inequality} we are left with equality due to Rankine--Hugoniot condition \eqref{Lagrange_Rankine_Hugoniot}. The same thing happens if $k = \mathcal U^+$ or $k=\mathcal U^-$.
\end{proof}

\begin{lemma}\label{lemma-U-entropy-int-ineq}
For all $k\in\mathbb{R}$ and every positive test function $\psi \in \mathcal D^+(Q_{lagr})$ with $\supp \psi$ containing only $\mathcal U$-shocks (and no $\zeta$-shocks), we have the entropy condition
\begin{align}
\label{entropy_integral_inequality}
\begin{split}
0 \leqslant &\iint\limits_{Q_{lagr}} |\mathcal U - k|\psi_x + \mathcal G(\mathcal U, k) \psi_\varphi \,d\varphi \,dx - \iint\limits_{Q_{lagr}} \mathcal F_\zeta(k, \zeta) \zeta_\varphi \sign(\mathcal U-k) \psi \,d\varphi \,dx \\
&{} + \int\limits_0^{x^0} \Big( |\mathcal U(\varphi_0(x), x) - k| s_0^x(x) + \mathcal G(\mathcal U(\varphi_0(x), x), k)\Big) \psi(\varphi_0(x), x) \, dx \\
&{} + \int\limits_{x^0}^\infty \mathcal G(\mathcal U(\varphi_0(x^0)+0, x), k)\psi(\varphi_0(x^0), x)\,dx + \int\limits_0^\infty |\mathcal U(\varphi, 0) - k|\psi(\varphi, 0)\,d\varphi.
\end{split}
\end{align}
\end{lemma}
\begin{proof}
Denote by 
\begin{align*}
\Gamma_\varphi &= \{(\varphi, 0)| \varphi\geqslant 0\}, \\ 
\Gamma_{\varphi x} &= \{(\varphi_0(x), x)| 0 < x < x^0\}, \\
\Gamma_x &= \{(\varphi_0(x^0), x)| x\geqslant x^0\}
\end{align*}
the boundaries of $Q_{lagr}$. Let $\omega_i$ be all the areas, where $\mathcal U$ is $\mathcal C^1$-smooth, $\Gamma_i^{\mathcal U}$ be all the $\mathcal U$-shocks other than $\Gamma_x$, $\Gamma_i^{\zeta}$ be all the $\zeta$-shocks.
Note first that 
\[
\mathcal U_x  + \mathcal F(\mathcal U, \zeta)_\varphi = 0 \quad\text{ in } \omega_i
\]
in the classical sense, and also 
\[
(\mathcal U - k + \mathcal F(\mathcal U, \zeta) - \mathcal F(k, \zeta)) \delta(\mathcal U - k) = 0
\]
because 
\[\mathcal U - k + \mathcal F(\mathcal U, \zeta) - \mathcal F(k, \zeta) = 0 \quad\text{ in } \supp \delta(\mathcal U - k).
\]
Therefore, inside the areas $\omega_i$ we have
\begin{align*}
|\mathcal U - k|_x + \mathcal G(\mathcal U, k)_\varphi & =  \Big(\mathcal U_x  + \mathcal F(\mathcal U, \zeta)_\varphi - \mathcal F(k, \zeta)_\varphi \Big) \sign(\mathcal U - k) \\
&\quad\,{} +
(\mathcal U - k + \mathcal F(\mathcal U, \zeta) - \mathcal F(k, \zeta)) 2\delta(\mathcal U - k) \mathcal U_\varphi \\
&{} = -\mathcal F_\zeta(k, \zeta) \zeta_\varphi \sign(\mathcal U - k),
\end{align*}
Thus, through integration by parts we obtain
\begin{align*}
\iint\limits_{\omega_i} |\mathcal U - k|\psi_x + \mathcal G(\mathcal U, k) \psi_\varphi \,d\varphi \,dx = & \int\limits_{\partial \omega_i} |\mathcal U - k|\psi n_x + \mathcal G(\mathcal U, k) \psi n_\varphi d\sigma \\
&{} + \iint\limits_{\omega_i} \mathcal F_\zeta(k, \zeta) \zeta_\varphi \sign(\mathcal U - k) \psi \, d\varphi\, dx,
\end{align*}
where $n = (n_\varphi, n_x)$ is the outwards normal at $\partial \omega_i$.
Summing integrals over all $\omega_i$ we obtain the full integrals over $Q_{lagr}$ and the sum of integrals over all shocks and boundaries. For $\zeta$-shocks $\Gamma_i^\zeta$ we note that $\psi = 0$ on them by our choice of test functions, therefore these integrals are zero as well. 

Over $\Gamma_\varphi$ we have $n_x = -1$ and $n_\varphi = 0$, therefore the integral simplifies into
\[
\int\limits_{\Gamma_\varphi} |\mathcal U - k|\psi n_x + \mathcal G(\mathcal U, k) \psi n_\varphi d\sigma = -\int\limits_{0}^\infty |\mathcal U(\varphi, 0+) - k| \psi(\varphi, 0) \, d\varphi,
\]
and since there cannot be a shock with infinite speed, we can resolve $\mathcal U(\varphi, 0+) = \mathcal U(\varphi, 0) = \mathcal U_0^\varphi(\varphi)$ to be the initial data for $\mathcal U$ on $\Gamma_\varphi$.

Over $\Gamma_x$ we have $n_x = 0$ and $n_\varphi = -1$, therefore the integral
\[
\int\limits_{\Gamma_x} |\mathcal U - k|\psi n_x + \mathcal G(\mathcal U, k) \psi n_\varphi d\sigma = -\int\limits_{x^0}^\infty \mathcal G(\mathcal U(\varphi_0(x^0)+, x), k) \psi(0, x) \, dx,
\]
but here we cannot resolve $\mathcal U(\varphi_0(x^0)+, x)$, since $\Gamma_x$ is a $\mathcal U$-shock with speed $0$. 

On $\Gamma_{\varphi x}$ we have $\varphi'_0(x) = -s_0^x(x)$, therefore
\[
n_x = \dfrac{-s_0^x(x)}{\sqrt{1 + (s_0^x(x))^2}}, \quad n_\varphi = \dfrac{-1}{\sqrt{1 + (s_0^x(x))^2}},
\]
\[
d\sigma = \sqrt{1 + (s_0^x(x))^2} \, dx,
\]
\begin{align*}
\int\limits_{\Gamma_{\varphi x}} &|\mathcal U - k|\psi n_x + \mathcal G(\mathcal U, k) \psi n_\varphi d\sigma = \\
{} &- \int\limits_0^{x^0} \Big( |\mathcal U(\varphi_0(x)+0, x) - k| s_0^x(x) + \mathcal G(\mathcal U(\varphi_0(x)+0, x), k)\Big) \psi(\varphi_0(x), x) \, dx.
\end{align*}
Here we resolve $\mathcal U(\varphi_0(x)+0, x) = \mathcal U(\varphi_0(x), x)$, since in the original coordinates there are no admissible shocks with speed zero, therefore there is at most finite number of discontinuity points on $[0, x^0] \times \{0\}$ in original coordinates and therefore on $\Gamma_{\varphi x}$ in Lagrange coordinates. 

\begin{figure}[H]
    \begin{minipage}{\linewidth}
    \begin{center}
    \def\lowborder{-1}
\def\leftborder{-1}
\def\upperborder{6}
\def\rightborder{10}

\begin{tikzpicture}[>=stealth', yscale=0.7, xscale=0.7]
\draw[thin,->] ({\leftborder}, 0) -- (\rightborder, 0) node[below] {$x$};
\draw[thin,->] (0,\lowborder) -- (0,\upperborder) node[below right ] {$\varphi$};
\draw[very thick] (\rightborder/2,\upperborder/2) ellipse [x radius=4,y radius=2];

\draw[very thick] (3, 0).. controls (5, 1) and (3, 2)  .. (7, 5.5) node[sloped,above, pos=0.8, minimum height = 20](N1){} node[sloped,below, pos=0.7, minimum height = 20](N2){} node[right, pos=1,]{$\Gamma_i^{\mathcal U}$};
\path (N1.south west) edge[-stealth, red] node[left,pos=0.9] {$n^-$} (N1.north west);
\path (N2.north west) edge[-stealth, red] node[right,pos=0.9] {$n^+$} (N2.south west);

\node[] at (\rightborder/ 4, \upperborder/2) {$\omega^+$};
\node[] at (\rightborder*3/ 4, \upperborder/2) {$\omega^-$};
\end{tikzpicture}
    \end{center}
    \end{minipage}
    \caption{Illustration for integrals over ${\partial\omega^+ \cap \Gamma_i^{\mathcal U}}$ and ${\partial\omega^- \cap \Gamma_i^{\mathcal U}}$.}
    \label{fig:lemma_6}
\end{figure}
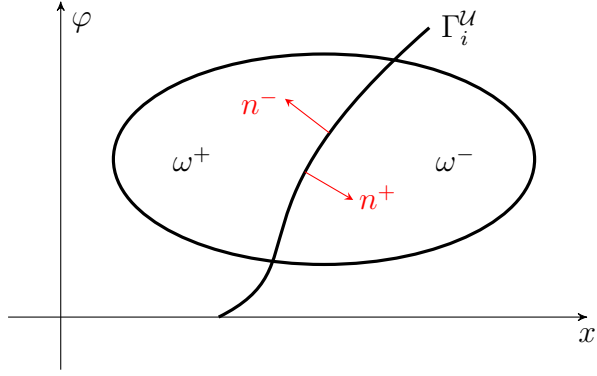

Finally, when we consider a $\mathcal U$-shock $\Gamma_i^{\mathcal U}$, we locally have two areas that bound it. Denote by $\varphi=\Phi(x)$ the curve of this shock, by $\omega^+$ the area where $\varphi > \Phi(x)$ and by $\omega^-$ the area where $\varphi < \Phi(x)$. Let $n^+ = (n^+_\varphi, n^+_x)$ and $n^- = (n^-_\varphi, n^-_x)$ be their outward normals at $\Gamma_i^{\mathcal U}$ (see Fig.~\ref{fig:lemma_6}). Note that by this definition we have
\[
n^-_\varphi = -n^+_\varphi, \quad
n^-_x = -n^+_x, \quad 
n^+_\varphi < 0, \quad 
\dfrac{d\Phi}{dx} = - \dfrac{n^\pm_x}{n^\pm_\varphi}.
\]
Therefore, due to the entropy inequality \eqref{entropy_inequality}, we conclude that
\begin{align*}
\int\limits_{\partial\omega^+ \cap \Gamma_i^{\mathcal U}} + \int\limits_{\partial\omega^- \cap \Gamma_i^{\mathcal U}} |\mathcal U - k|\psi n_x &+ \mathcal G(\mathcal U, k) \psi n_\varphi d\sigma \\
&{} =
\int\limits_{\Gamma_i^{\mathcal U}} n_\varphi^+ \Big( [\mathcal G(\mathcal U, k)] - \dfrac{d\Phi}{dx}[|\mathcal U - k|]  \Big) \psi \, d\sigma \geqslant 0.
\end{align*}
Combining all terms of the sum we arrive at \eqref{entropy_integral_inequality}.
\end{proof}

The next proposition follows from Corollary \ref{corollary:zero_flow_area} and is a requirement in the proof of Kru\v{z}kov's theorem.

\begin{proposition}
Any solution $\mathcal U$ corresponding to an admissible solution in original coordinates belongs to the class $\mathcal U \in L^\infty_{loc}(Q_{lagr})$.
\end{proposition}

Using this restriction on $\mathcal U$ it is possible to prove the following variation of the entropy inequality.

\begin{proposition}[Proposition 2.7.1, \cite{Serre1}]
\label{proposition_entropy_solution_inequality}
For a fixed solution $\zeta$, given two solutions $\mathcal U$ and $\mathcal V$, for any positive test function $\psi \in \mathcal D^+(Q_{lagr})$ with $\supp \psi$ containing no $\zeta$-shocks, we have
\begin{align}
\label{entropy_solutions_inequality}
\begin{split}
0 \leqslant &\iint\limits_{Q_{lagr}} |\mathcal U - \mathcal V|\psi_x + \mathcal G(\mathcal U, \mathcal V) \psi_\varphi \,d\varphi \,dx \\
&{} + \int\limits_0^{x^0} |\mathcal U(\varphi_0(x), x) - \mathcal V(\varphi_0(x), x)| s_0^x(x) \psi(\varphi_0(x), x) \, dx \\
&{} + \int\limits_0^{x^0} \mathcal G(\mathcal U(\varphi_0(x), x), \mathcal V(\varphi_0(x), x)) \psi(\varphi_0(x), x) \, dx \\
&{} + \int\limits_{x^0}^\infty \mathcal G(\mathcal U(\varphi_0(x^0)+0, x), \mathcal V(\varphi_0(x^0)+0,x))\psi(0, x)\,dx \\
&{} + \int\limits_0^\infty |\mathcal U(\varphi, 0) - \mathcal V(\varphi, 0)|\psi(\varphi, 0)\,d\varphi.
\end{split}
\end{align}
\end{proposition}

In \cite{Serre1} this Proposition is proved for the case that corresponds to $\zeta \equiv const$. For the general case, see the original proof of Kru\v{z}kov's \cite[Theorem 1]{Kruzhkov}.

\subsection{Oleinik and Lax admissibility of $\zeta$-shocks}

Recall that due to Proposition \ref{prop:inadmissible_shocks} (and also Proposition \ref{prop:c-shock-admissibility}) for an admissible shock with different values of $c$ we must have $c^- > c^+$ and thus $\zeta^- < \zeta^+$. Therefore, due to the concavity of the function $a$ we always have Lax condition for the equation \eqref{eq:c-lagr-eqn}.
\begin{proposition}
For any admissible $\zeta$-shock we have Lax condition
\[
a_\zeta(\zeta^+) < \dfrac{a(\zeta^-)-a(\zeta^+)}{\zeta^--\zeta^+} < a_\zeta(\zeta^-).
\]
Moreover, 
\[
\dfrac{a(\zeta)-a(\zeta^-)}{\zeta-\zeta^-} > \dfrac{a(\zeta^-)-a(\zeta^+)}{\zeta^--\zeta^+} \quad \text{ for all } \zeta\in(\zeta^-, \zeta^+),
\]
therefore, Oleinik's E-condition also holds.
\end{proposition}
Similar to Lemma \ref{lemma-U-entropy-ineq} and Lemma \ref{lemma-U-entropy-int-ineq} we obtain the following entropy inequalities from Oleinik's E-condition.
\begin{proposition}
\label{prop-zeta-entropy-integral}
Denote $\mathcal A(\zeta, k) = (a(\zeta) - a(k))\sign(\zeta-k)$. Then on any admissible $\zeta$-shock $\Phi(x)$ inside $Q_{lagr}$ we have the entropy inequality
\[
[\mathcal A(\zeta, k)] \leqslant \dfrac{d\Phi}{dx}[|\zeta-k|], \quad k\in\mathbb{R},
\]
and therefore for every positive test function $\psi \in \mathcal D^+(Q_{lagr})$ we have the integral entropy condition
\begin{align*}
0 \leqslant &\iint\limits_{Q_{lagr}} |\zeta - k|\psi_x + \mathcal A(\zeta, k)\psi_\varphi \,d\varphi\,dx \\
&{} + \int\limits_0^{x^0} \Big( |\zeta(\varphi_0(x), x) - k| s_0^x(x) + \mathcal A(\zeta(\varphi_0(x), x), k)\Big) \psi(\varphi_0(x), x) \, dx \\
{} & + \int\limits_{x^0}^\infty \mathcal A(\zeta(\varphi_0(x^0),x), k)\psi(\varphi_0(x^0),x) \,dx + \int\limits_0^\infty |\zeta(\varphi, 0) - k|\psi(\varphi, 0)\, d\varphi.
\end{align*}
\end{proposition}

Unfortunately, the entropy inequality \eqref{entropy_inequality} does not hold on a $\zeta$-shock. Clearly, for $k < \min(\mathcal U^+, \mathcal U^-)$ in the case $(\mathcal F 4.1)$ due to Rankine--Hugoniot condition we have
\[
[\mathcal G(\mathcal U, k)] - \dfrac{d\Phi}{dx} [|\mathcal U - k|] = \mathcal F(k, \zeta^-) - \mathcal F(k, \zeta^+) > 0,
\]
and similarly, for $k > \max(\mathcal U^+, \mathcal U^-)$ in the case $(\mathcal F 4.2)$ we have
\[
[\mathcal G(\mathcal U, k)] - \dfrac{d\Phi}{dx} [|\mathcal U - k|] = \mathcal F(k, \zeta^+) - \mathcal F(k, \zeta^-) > 0.
\]
Thankfully, even though such shocks are not entropic in the same sense as $\mathcal U$-shocks, they still satisfy the integral inequality similar to Proposition~\ref{proposition_entropy_solution_inequality} for any pair of solutions with admissible shocks. 

\begin{lemma}
\label{lemma_zeta_shock_entropy_condition}
For a fixed solution $\zeta$, given two solutions $\mathcal U$ and $\mathcal V$, we have the following entropy inequality on any $\zeta$-shock:
\begin{equation}\label{eq:entropy-ineq-zeta-shock}
[\mathcal G(\mathcal U, \mathcal V)] \leqslant \dfrac{d\Phi}{dx} [|\mathcal U - \mathcal V|].
\end{equation}
\end{lemma}
\begin{proof}
In order to prove this inequality, we analyze the nullcline configurations for $\mathcal U$ and $\mathcal V$ at any given point of the $\zeta$-shock. We assume that both solutions have a Type I nullcline configuration, all other cases are analyzed similarly. We also assume that the nullcline configurations for these solutions do not coincide. If they do, then we have equality in \eqref{eq:entropy-ineq-zeta-shock} following immediately from Rankine--Hugoniot condition.

\begin{figure}[ht!]
    \begin{center}
    \includegraphics[trim={0.1cm 0.1cm 0.1cm 0.1cm}, clip, width=0.55\linewidth]{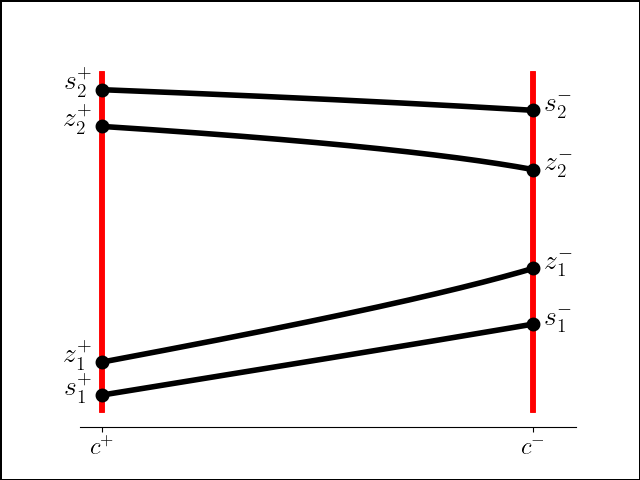}
    \end{center}
    \caption{An illustration for the positions of $z^\pm_{1,2}$ and $s^\pm_{1,2}$ on the nullcline configuration plane.}
    \label{fig:two_nullcline_configurations}
\end{figure}
On Fig.\ref{fig:two_nullcline_configurations} critical points $s^\pm_{1,2}$ contain $\vartheta_{\zeta^\mp}(\mathcal U^\mp)$ (recall the definition of $\vartheta_{\zeta^\mp}(\mathcal U^\mp)$ from \eqref{shock_mapping}) and correspond to the shock speed $v$, and critical points $z^\pm_{1,2}$ contain $\vartheta_{\zeta^\mp}(\mathcal V^\mp)$ and correspond to the speed $w$. As mentioned above, if $v = w$, then the entropy inequality \eqref{eq:entropy-ineq-zeta-shock} immediately follows from the Rankine--Hugoniot condition. Therefore, without loss of generality, we assume $v<w$. Now we can break down all possible cases for how the values $\vartheta_{\zeta^\mp}(\mathcal V^\mp)$ are arranged with respect to $\vartheta_{\zeta^\mp}(\mathcal U^\mp)$.
\begin{itemize}
    \item $\mathcal U^+ = \theta_{c^-}(s_2^-)$, $\mathcal U^- = \theta_{c^+}(s_1^+)$. This case corresponds to an inadmissible shock due to Proposition \ref{prop:c-shock-admissibility}, therefore we don't need to consider it.
    \item $\mathcal U^+ = \theta_{c^-}(s_2^-)$, $\mathcal U^- = \theta_{c^+}(s_2^+)$ or $\mathcal U^+ = \theta_{c^-}(s_1^-)$, $\mathcal U^- = \theta_{c^+}(s_1^+)$. In these cases we have 
    \[
    \sign(\mathcal U^+ - \mathcal V^+) = \sign(\mathcal U^- - \mathcal V^-),
    \]
    therefore we have equality in \eqref{eq:entropy-ineq-zeta-shock} due to Rankine--Hugoniot condition \eqref{Lagrange_Rankine_Hugoniot}.
    \item $\mathcal U^+ = \theta_{c^-}(s_1^-)$, $\mathcal U^- = \theta_{c^+}(s_2^+)$. In this case 
    \begin{equation}
    \label{eq:known_signs}
    \sign(\mathcal U^+ - \mathcal V^+) = 1, \quad \sign(\mathcal U^- - \mathcal V^-) = -1,
    \end{equation}
    so \eqref{eq:entropy-ineq-zeta-shock} transforms into
    \[
    \mathcal F(\mathcal U^+, \zeta^+) - \mathcal F(\mathcal V^+, \zeta^+) + \mathcal F(\mathcal U^-, \zeta^-) - \mathcal F(\mathcal V^-, \zeta^-) \leqslant \dfrac{d\Phi}{dx} (\mathcal U^+ - \mathcal V^+ + \mathcal U^- - \mathcal V^-).
    \]
    Note that due to the monotonicity of $a$ and the Rankine-Hugoniot condition \eqref{eq:RH-2} the speed of the shock
    \begin{equation}
    \label{eq:zeta-shock-velocity-positive}
    \dfrac{d\Phi}{dx} = \dfrac{a(\zeta^-)-a(\zeta^+)}{\zeta^--\zeta^+} > 0.
    \end{equation}
    It could be observed geometrically on Fig.~\ref{fig:two_nullclines_hodograph} that the following inequalities hold for the inclines of lines:
    \[
    \dfrac{\mathcal F(\mathcal U^+, \zeta^+) - \mathcal F(\mathcal V^+, \zeta^+)}{\mathcal U^+ - \mathcal V^+} < \dfrac{d\Phi}{dx},
    \]
    \[
    \dfrac{\mathcal F(\mathcal U^-, \zeta^-) - \mathcal F(\mathcal V^-, \zeta^-)}{\mathcal U^- - \mathcal V^-} > \dfrac{d\Phi}{dx}.
    \]
    Taking into account the known signs \eqref{eq:known_signs} of the denominators, we obtain
     \[
    \mathcal F(\mathcal U^+, \zeta^+) - \mathcal F(\mathcal V^+, \zeta^+) < \dfrac{d\Phi}{dx} (\mathcal U^+ - \mathcal V^+),
    \]
    \[
    \mathcal F(\mathcal U^-, \zeta^-) - \mathcal F(\mathcal V^-, \zeta^-) < \dfrac{d\Phi}{dx} (\mathcal U^- - \mathcal V^-).
    \]
    Taking the sum of these inequalities concludes the proof.
\end{itemize}
\begin{figure}[t!]
    \begin{center}
    \includegraphics[trim={0.1cm 0.1cm 0.1cm 0.1cm}, clip, width=0.8\linewidth]{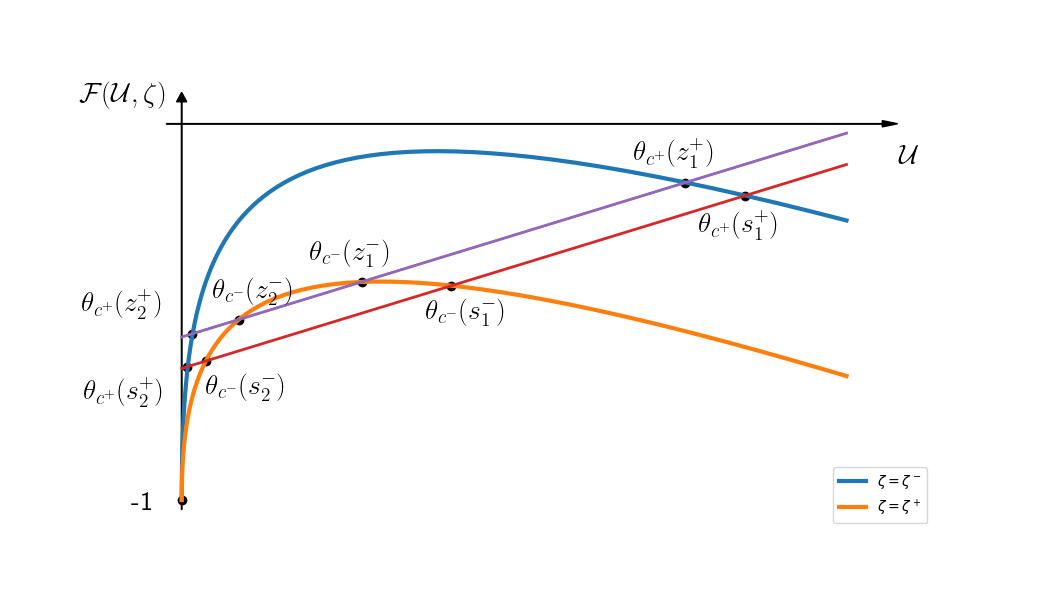}
    \end{center}
    \caption{An illustration for the positions of $\theta_{c^\pm}(z^\pm_{1,2})$ and $\theta_{c^\pm}(s^\pm_{1,2})$ on the plots of $\mathcal F(\cdot, \zeta^\pm)$.}
    \label{fig:two_nullclines_hodograph}
\end{figure}
\end{proof}

Note that due to \eqref{eq:zeta-shock-velocity-positive} any  $\zeta$-shock has at most one intersection point with the boundaries of $Q_{lagr}$. Thus, the integrals over them are omitted in the following lemma.

\begin{lemma}\label{lemma_entropy_ineq_zeta_shock}
For a fixed solution $\zeta$, consider a $\zeta$-shock $\Gamma^\zeta$ given by the curve $\Phi(x)$. Denote by $\psi^\varepsilon(\varphi, x)$ a series of positive test functions with supports contracting towards $\Gamma^\zeta$:
\[
\psi^\varepsilon(\varphi, x) = \psi(\varphi, x) (1 - \Theta_\varepsilon (\dist((\varphi, x), \{(\Phi(x), x)\})),
\] where $\psi\in\mathcal D^+(Q_{lagr})$, $\dist(.,.)$ is the distance function, and
\[
\Theta_\varepsilon(r) = \begin{cases}
    0, & r\leqslant 0, \\
    1, & r\geqslant \varepsilon,
\end{cases}
\]
is a smooth approximation of the Heaviside function.

Then, given two solutions $\mathcal U$ and $\mathcal V$, the entropy inequality holds in the limit:
\[
0 \leqslant \lim\limits_{\varepsilon\to 0}\iint\limits_{Q_{lagr}} |\mathcal U - \mathcal V|\psi^\varepsilon_x + \mathcal G(\mathcal U, \mathcal V) \psi^\varepsilon_\varphi \,d\varphi \,dx.
\]
\end{lemma}
\begin{proof}
Similar to Lemma \ref{lemma-U-entropy-int-ineq} we begin with a relation inside the smoothness areas:
\begin{align*}
|\mathcal U - \mathcal V|_x + \mathcal G(\mathcal U, \mathcal V)_\varphi & =  \Big(\mathcal U_x - \mathcal V_x  + \mathcal F(\mathcal U, \zeta)_\varphi - \mathcal F(\mathcal V, \zeta)_\varphi \Big) \sign(\mathcal U - \mathcal V) \\
&\quad\,{} +
(\mathcal U - \mathcal V + \mathcal F(\mathcal U, \zeta) - \mathcal F(\mathcal V, \zeta)) 2\delta(\mathcal U - \mathcal V) (\mathcal U_\varphi - \mathcal V_\varphi) = 0.
\end{align*}
Now it is clear that integrating by parts we obtain
\begin{align*}
\lim\limits_{\varepsilon\to 0}\iint\limits_{Q_{lagr}} & |\mathcal U - \mathcal V|\psi^\varepsilon_x + \mathcal G(\mathcal U, \mathcal V) \psi^\varepsilon_\varphi \,d\varphi \,dx \\
&{} = \int\limits_{\Gamma^{\zeta}_+} + \int\limits_{\Gamma^{\zeta}_-} |\mathcal U - \mathcal V|\psi n_x + \mathcal G(\mathcal U, \mathcal V) \psi n_\varphi d\sigma \\
&{} =
\int\limits_{\Gamma^{\zeta}} n_\varphi^+ \Big( [\mathcal G(\mathcal U, \mathcal V)] - \dfrac{d\Phi}{dx}[|\mathcal U - \mathcal V|]  \Big) \psi \, d\sigma \geqslant 0
\end{align*}
due to Lemma \ref{lemma_zeta_shock_entropy_condition}. Here $\Gamma^{\zeta}_+$ and $\Gamma^{\zeta}_-$ denote respectively the border of the area above $\Gamma^\zeta$ and below it, with respective change in normal sign.
\end{proof}

\begin{lemma}

Proposition \ref{proposition_entropy_solution_inequality} holds for all positive test functions $\psi \in \mathcal D^+(Q_{lagr})$ without restrictions on their supports.

\end{lemma}
\begin{proof}
We unite all $\zeta$-shocks $\Gamma^\zeta_i$ and consider
\[
\varpi_\varepsilon(\varphi, x) = \Theta_\varepsilon(\dist((\varphi, x), \cup_i \Gamma^\zeta_i)),
\]
where $\Theta_\varepsilon$ is defined in Lemma \ref{lemma_entropy_ineq_zeta_shock}. Any given test function $\psi \in \mathcal D^+(Q_{lagr})$ we decompose into
\[
\psi(\varphi, x) = \psi^\varepsilon_\zeta(\varphi, x) + \psi^\varepsilon_{\mathcal U}(\varphi, x),
\]
where
\begin{align*}
\psi^\varepsilon_\zeta(\varphi, x) &= \psi(\varphi, x) (1 - \varpi_\varepsilon(\varphi, x)), \\
\psi^\varepsilon_{\mathcal U}(\varphi, x) &=
\psi(\varphi, x) \varpi_\varepsilon(\varphi, x).
\end{align*}

By Proposition \ref{proposition_entropy_solution_inequality} we have \eqref{entropy_solutions_inequality} for $\psi^\varepsilon_{\mathcal U}(\varphi, x)$ for all $\varepsilon > 0$, and by Lemma \ref{lemma_entropy_ineq_zeta_shock} it holds for $\psi^\varepsilon_{\zeta}(\varphi, x)$ in the limit as $\varepsilon \to 0$. Taking the sum of these two inequalities we prove this lemma.
\end{proof}

\section{The uniqueness theorem}
\label{sec:uniqueness}

In his famous 1970 paper \cite{Kruzhkov} Kru\v{z}kov proved the weak solution uniqueness theorem for scalar conservation laws. We will follow the scheme of his proof as explained in \cite{Serre1} to prove a similar uniqueness theorem for the polymer injection system \eqref{eq:main_system_chem_flood}.

\begin{theorem}
Problem \eqref{eq:main_system_chem_flood} with initial-boundary conditions \eqref{eq:Initial_boundary_problem} satisfying the restrictions (S1)--(S3), with flow function satisfying (F1)--(F4) and adsorption satisfying (A1)--(A3) can only have a unique piece-wise $\mathcal C^1$-smooth weak solution with vanishing viscosity admissible shocks and locally bounded ``variation'' of $c$.
\end{theorem}
\begin{proof}
To prove this theorem, we first follow Section \ref{sec2-Lagrange} to transform system \eqref{eq:main_system_chem_flood} into the Lagrange coordinate system \eqref{eq:U-lagr-eqn}, \eqref{eq:c-lagr-eqn}. Under our restrictions, this transformation gives a one-to-one mapping of solutions between original and Lagrange coordinates. Therefore, by demonstrating the uniqueness of the solution in Lagrange coordinates we prove this theorem.

Equations \eqref{eq:U-lagr-eqn} and \eqref{eq:c-lagr-eqn} are decoupled. Equation \eqref{eq:c-lagr-eqn} does not depend on $\mathcal U$ and therefore could be solved as a scalar conservation law. Due to Proposition \ref{prop-zeta-entropy-integral}, original Kru\v{z}kov's theorem is applicable with minimal adaptations for the shape of $Q_{lagr}$ and the form of the boundary conditions. Therefore, the solution for $\zeta$ is unique.

Now suppose there are two different solutions $\mathcal U$ and $\mathcal V$ of the first equation \eqref{eq:U-lagr-eqn} with fixed unique $\zeta$. Then we substitute them into \eqref{entropy_solutions_inequality}. The terms containing boundary and initial data vanish. Indeed, since both functions solve the same initial-boundary problem, functions $s_0^x(x)$ and $s_0^t(t)$ are the same for both solutions, therefore
\begin{align*}
& \mathcal U(\varphi, 0) = \mathcal V(\varphi, 0) & & \forall \varphi\geqslant 0, \\
&\mathcal U(\varphi_0(x), x) = \mathcal V(\varphi_0(x), x) & & \forall x\in(0, x^0).
\end{align*}
Moreover, due to Proposition \ref{prop:boundary_shock_oleinik} and ($\mathcal F$2) (recall Proposition \ref{prop:new-flow-function-properties}) we have
\[
\mathcal G(\mathcal U(\varphi_0(x), x), \mathcal V(\varphi_0(x), x)) \leqslant 0, \qquad \forall x \geqslant x^0.
\]
Therefore, 
\[
\int\limits^\infty_{x^0} \mathcal G(\mathcal U(\varphi_0(x^0)+0, x), \mathcal V(\varphi_0(x^0)+0, x)) \psi(\varphi_0(x^0), x) \, dx \leqslant 0,
\]
and we can eliminate almost everything in \eqref{entropy_solutions_inequality} and arrive at
\[
\iint\limits_{Q_{lagr}} |\mathcal U - \mathcal V|\psi_x + \mathcal G(\mathcal U, \mathcal V) \psi_\varphi \,d\varphi \,dx \geqslant 0.
\]

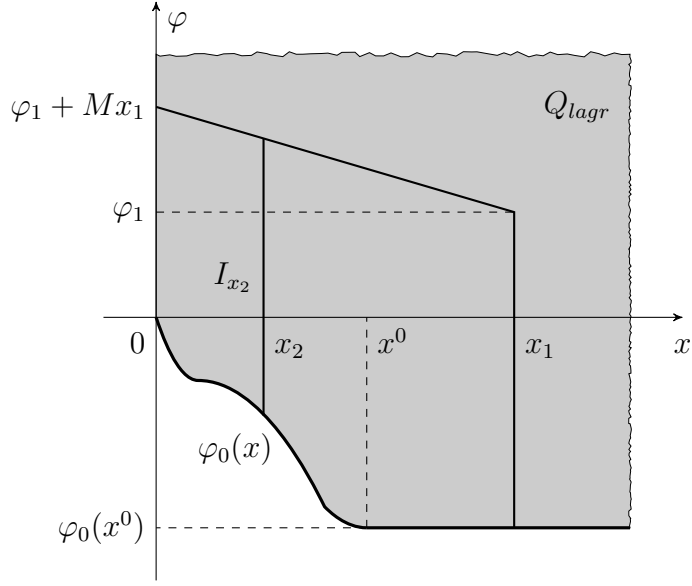
\begin{figure}[ht!]
    \begin{minipage}{\linewidth}
    \begin{center}
    \def\lowborder{-5}
\def\leftborder{-1}
\def\upperborder{6}
\def\rightborder{10}
\def\startpoint{\rightborder * 0.4}
\def\phinote{0.8 * \lowborder}
\def\eps{1}

\begin{tikzpicture}[>=stealth', yscale=0.7, xscale=0.7]
	\begin{scope}
		\filldraw[fill=black!20, ultra thin]
		(0, 0)  parabola[bend at end] (\startpoint * 0.2, \phinote * 0.3)
		parabola (\startpoint * 0.8, \phinote *0.9) parabola[bend at end] (\startpoint, \phinote) -- (\rightborder - \eps, \phinote)
		decorate [decoration={random steps,segment length=1pt,amplitude=0.5pt}] {-- (\rightborder - \eps, \upperborder - \eps)}
		decorate [decoration={random steps,segment length=3pt,amplitude=1pt}] {--(0, \upperborder - \eps)}
		-- (0, 0) -- cycle;
		\draw[thin,->] ({\leftborder}, 0) -- (\rightborder, 0) node[below] {$x\vphantom{x^0}$};
		\draw[very thick, black] (0, 0)  parabola[bend at end] (\startpoint * 0.2, \phinote * 0.3)
		parabola (\startpoint * 0.8, \phinote *0.9) parabola[bend at end] (\startpoint, \phinote);
		\node[] at (\rightborder * 0.8, 4) {$Q_{lagr}$};
		\node[below right] at (\startpoint, 0) {$x^0$};
		\node[below left] at (0, 0) {$0\vphantom{x^0}$};
		\draw[very thick, black] (\startpoint, \phinote) -- (\rightborder - \eps, \phinote);
		\draw[dashed]  (\startpoint, \phinote) -- (\startpoint, 0);
		\draw[dashed] (\startpoint, \phinote) -- (0, \phinote);
		\node[left] at (0, \phinote) {$\varphi_0(x^0  )$};
		\node[below] at (\startpoint/2-0.5, \phinote/2) {$\varphi_0(x)$};
	\end{scope}
	\draw[thin,->] (0,\lowborder) -- (0,\upperborder) node[below right ] {$\varphi$};
	\node[left] at (0, \upperborder/3) {$\varphi_1$};
	\node[left] at (0, \upperborder*2/3) {$\varphi_1 + Mx_1$};
 	\node[below right] at (\startpoint *1.7, 0) {$x_1\vphantom{x^0}$};
	\draw[thick] (0, \upperborder*2/3) -- (\startpoint *1.7, \upperborder/3) node[pos=0.3, above](Ipoint){} -- (\startpoint *1.7, \phinote);
	\draw[dashed] (0, \upperborder/3) -- (\startpoint *1.7, \upperborder/3);
	\draw[thick] (Ipoint) -- (\startpoint *1.7*0.3, -1.87) node[left, pos=0.5]{$I_{x_2}$};
    \node[below right] at (\startpoint *1.7*0.3, 0) {$x_2\vphantom{x^0}$};
\end{tikzpicture}
    \end{center}
    \end{minipage}
    \caption{The section $I_{x_2}$ of the domain of integration.}
    \label{fig:thm1-pic}
\end{figure}

Denote by $M = -\min \mathcal{F}_{\mathcal U}$. Fix a point $(\varphi_1, x_1) \in Q_{lagr}$. We assume $\varphi_1 > 0, x_1 > x^0$, other cases are similar. Consider a cut-off function $$\varpi(\varphi, x) = \Theta_\varepsilon(\varphi_1 - \varphi + M(x_1-x))$$ and an arbitrary scalar function $\chi(x) \in \mathcal D^+([0, \infty))$. Construct a positive test function
\[
\psi(\varphi, x) = \chi(x) \varpi(\varphi, x).
\]
Expand
\begin{align*}
|\mathcal U &- \mathcal V|\psi_x + \mathcal G(\mathcal U, \mathcal V) \psi_\varphi \\
&= \chi'(x) \varpi(\varphi, x) |\mathcal U - \mathcal V| - \chi(x) \Theta'_\varepsilon(\varphi_1 - \varphi + M(x_1-x)) \Big( M |\mathcal U - \mathcal V| + \mathcal G(\mathcal U, \mathcal V) \Big).
\end{align*}
Note that $\Theta'_\varepsilon \geqslant 0$, and due to the choice of $M$ we have
\[
M |\mathcal U - \mathcal V| + \mathcal G(\mathcal U, \mathcal V) \geqslant 0.
\]
Therefore, we arrive at
\[
\iint\limits_{Q_{lagr}} \chi'(x) \Theta_\varepsilon(\varphi_1 - \varphi + M(x_1-x)) |\mathcal U - \mathcal V| \,d\varphi \,dx \geqslant 0.
\]
Now we denote $I_x = (\varphi_0(\min(x, x^0)), \varphi_1 + M(x_1 - x))$ the intervals in the vertical sections of the limiting ``support'' of $\psi$ (see an example $I_{x_2}$ on Fig.~\ref{fig:thm1-pic}) and
\[
h(x) = \int\limits_{I_x} |\mathcal U(\varphi, x) - \mathcal V(\varphi, x)| \, d\varphi.
\]
As we tend $\varepsilon \to 0$, we arrive at
\[
\int\limits_0^\infty \chi'(x) h(x) \, dx \geqslant 0
\]
for any positive test function $\chi$. Since we have $h(0) = 0$ from the initial data, $h$ is classically decreasing (integrating by parts we arrive at $h'\leqslant 0$ as a generalized function), and we arrive at 
\[
h(x_1) = \int\limits_{\varphi_0(x^0)}^{\varphi_1} |\mathcal U(\varphi, x_1) - \mathcal V(\varphi, x_1)| \, d\varphi = 0 \qquad \forall \varphi_1 > 0, x_1 > x^0.
\]
Similarly, for $x_1 \leqslant x^0$ we obtain
\[
h(x_1) = \int\limits_{\varphi_0(x_1)}^{\varphi_1} |\mathcal U(\varphi, x_1) - \mathcal V(\varphi, x_1)| \, d\varphi = 0 \qquad \forall \varphi_1 > 0, x_1 \leqslant x^0.
\]
and we conclude that $\mathcal U = \mathcal V$.

\end{proof}

\section{Discussions}
\label{sec:discussions}

As mentioned above, our proposed form of admissibility is convenient for applications, since it is easier to verify, than classical vanishing viscosity. However, the restrictions listed in Definition \ref{def:solution} somewhat limit the scope of the result. Therefore, in future works we aim to lift or weaken every restriction we can.

Future research may include:
\begin{itemize}
    \item Application of the theorem to some known solutions (e.g. polymer slug) to verify and adjust the proposed heuristics used in their construction (notably, the Jouguet principle used to construct characteristics in the area where characteristics from the boundaries do not arrive).
    \item Weakening or eliminating the restriction (W2).
    \item Derivation of the properties (W3) and (W4) from the classical form of the vanishing viscosity condition. 
    \item Investigation into the existence of initial-boundary conditions that lead to the formation of a locally infinite number of shocks.
    \item Investigation into weakening or lifting the restrictions on the initial-boundary data (S1)--(S3).
    \item Generalizing the theorem to more general classes of $f$ and $a$ without some of the restrictions (F1)--(F4), (A1)--(A3).
    \item Similar results for other systems of two equations that allow for the splitting of equations under an appropriate Lagrange coordinate transformation (e.g. thermal problems with the second equation governing the transport of heat).
\end{itemize}

\section*{Acknowledgements}

The authors thank Pavel Bedrikovetsky for lectures on systems of hyperbolic conservation laws.
Research is supported by the Russian Science Foundation (RSF) grant 19-71-30002.
\bigskip
\bigskip

\begin{enbibliography}{99}
\addcontentsline{toc}{section}{References}

\bibitem{Pires2021}
Apolin\'{a}rio, F.O. and Pires, A.P., 2021. Oil displacement by multicomponent slug injection: An analytical solution for Langmuir adsorption isotherm. Journal of Petroleum Science and Engineering, 197, p.~107939.

\bibitem{Bahetal}
Bakharev, F., Enin, A., Petrova, Y. and Rastegaev, N., 2023. Impact of dissipation ratio on vanishing viscosity solutions of the Riemann problem for chemical flooding model. Journal of Hyperbolic Differential Equations, 20(02), pp.~407-432.

\bibitem{Tapering}
Bakharev, F., Enin, A., Kalinin, K., Petrova, Y., Rastegaev, N. and Tikhomirov, S., 2023. Optimal polymer slugs injection profiles. Journal of Computational and Applied Mathematics, 425, p.~115042.

\bibitem{BL} Buckley, S.~E. and Leverett, M., 1942. Mechanism of fluid displacement in sands. Transactions of the AIME, 146(01), pp.~107-116.

\bibitem{Castaneda}
Casta\~{n}eda, P., 2016. Dogma: S-shaped. The Mathematical Intelligencer, 38, pp.~10-13.

\bibitem{Courant}
Courant, R., 1944. Supersonic Flow and Shock Waves: A Manual on the Mathematical Theory of Non-linear Wave Motion (No. 62). Courant Institute of Mathematical Sciences, New York University.

\bibitem{Gelfand}
Gelfand, I.~M., 1959. Some problems in the theory of quasilinear equations. Uspekhi Matematicheskikh Nauk, 14(2), pp.~87-158 (in Russian). English translation in Transactions of the American Mathematical Society, 29(2), 1963, pp.~295-381.

\bibitem{JnW}
Johansen, T. and Winther, R., 1988. The solution of the Riemann problem for a hyperbolic system of conservation laws modeling polymer flooding. SIAM journal on mathematical analysis, 19(3), pp.~541-566.

\bibitem{Kruzhkov}
Kru\v{z}kov, S.N., 1970. First order quasilinear equations in several independent variables. Mathematics of the USSR-Sbornik, 10(2), pp.~217-243.

\bibitem{Oleinik}
Oleinik, O.~A., 1957. Discontinuous solutions of non-linear differential equations. Uspekhi Matematicheskikh Nauk, 12(3)(75), pp.~3-73 (in Russian). English translation in American Mathematical Society Translations, 26(2), 1963, pp.~95-172.

\bibitem{PiBeSh06}
Pires, A.P., Bedrikovetsky, P.G. and Shapiro, A.A., 2006. A splitting technique for analytical modelling of two-phase multicomponent flow in porous media. Journal of Petroleum Science and Engineering, 51(1-2), pp.~54-67.

\bibitem{RastS-Shaped}
Rastegaev, N., 2023. On the sufficient conditions for the S-shaped Buckley-Leverett function. arXiv preprint arXiv:2303.16803.

\bibitem{Serre1}
Serre, D. Systems of Conservation Laws 1: Hyperbolicity, entropies, shock waves. Cambridge University Press, 1999.

\bibitem{Shen}
Shen, W., 2017. On the uniqueness of vanishing viscosity solutions for Riemann problems for polymer flooding. Nonlinear Differential Equations and Applications NoDEA, 24, pp.~1-25.

\bibitem{Wa87}
Wagner, D.H., 1987. Equivalence of the Euler and Lagrangian equations of gas dynamics for weak solutions. Journal of differential equations, 68(1), pp.~118-136. 

\end{enbibliography}

\end{document}